\documentclass[reqno,12pt]{amsart}
\usepackage{a4,latexsym,amssymb,amsfonts}

\usepackage{amsmath,amscd} 

\usepackage{
pdfsync,
verbatim} 

\usepackage{enumerate}
\numberwithin{equation}{section}
\usepackage{amsthm}

\newtheorem{theorem}{Theorem}[section]
\newtheorem{lemma}[theorem]{Lemma}
\newtheorem{proposition}[theorem]{Proposition}

\newtheorem{definition}[theorem]{Definition}

\theoremstyle{definition}
\newtheorem{remark}[theorem]{Remark}

\newcommand\II{{\mathbb{I}}}
\newcommand\RR{{\mathbb{R}}}
\newcommand\CC{{\mathbb{C}}}
\newcommand\NN{{\mathbb{N}}}
\newcommand\ZZ{{\mathbb{Z}}}
\newcommand\HH{{\mathbb{H}}}
\newcommand\JJ{{\mathbb{J}}}

\newcommand\cP{{\mathcal{P}}}

\newcommand\cS{{\mathcal{S}}}

\date{today}

\begin{document}

\title[A Restriction Theorem for M\'etivier
Groups]{A Restriction Theorem for  M\'etivier Groups}
\author{Valentina Casarino and Paolo Ciatti}
\address{DTG,
Universit\`a degli Studi di Padova\\
Stradella san Nicola 3\\
I-36100 Vicenza}
\email{valentina.casarino@unipd.it}
\address{
Dicea, Universit\`a degli Studi di Padova\\
via Marzolo 9\\ I-35131 Padova}
\email{paolo.ciatti@unipd.it}
\thanks{It is a pleasure to thank the anonymous
referees for carefully reading the manuscript  and for providing 
thoughtful and 
constructive comments.}
%22E25 Nilpotent and solvable Lie groups
%22E30 Analysis on real and complex Lie groups
%43A80 Analysis on other specific Lie groups
\date{\today}

\maketitle

\begin{abstract}
In the spirit of an
earlier result of D. M\"uller
on the Heisenberg group
we prove a restriction theorem on a
certain class of two step nilpotent
Lie groups.
Our result extends
that of M\"uller also
in the framework of the 
Heisenberg group.
\end{abstract}

\date{today}

\section{Introduction}

In this paper we examine the mapping properties  between Lebesgue spaces of the operators arising in the spectral resolution of the sublaplacian
on the class of groups 
introduced by G. M\'etivier  in \cite{Me}.
These are two-step nilpotent 
Lie groups, characterized by the property that
the quotients with respect to the hyperplanes
contained in the centre are
Heisenberg groups.
The groups of $H$-type, introduced
by A. Kaplan \cite{K}, 
are examples of groups
satisfying the M\'etivier property,
but there are M\'etivier groups
which are not of $H$-type
(for an example see \cite{MuS}).

Let $G$ be a  M\'etivier group
equipped with a sublaplacian $L$.
The operators
$\mathcal P^L_{\mu}$,
we are interested in, are formally given by a Dirac delta $\delta_{\mu}(L)$
at a point ${\mu}$ of the spectrum of $L$.
$\mathcal P^L_{\mu}$ corresponds to a ``generalized projection operator"
in the  sense introduced by R. Strichartz in  \cite{Str}: its range consists of eigenfunctions of $L$,
and the family 
$\{\mathcal P^L_{\mu}\}$
decomposes $f$, in the sense that
for all Schwartz functions $f$ on $G$
$$f=\int_0^{+\infty}
\mathcal P^L_{\mu}f\,d\mu \,,$$
where the integral converges in the sense of distributions,
as we shall prove in Theorem \ref{risoluzione-Strichartz}.

On the Euclidean
space $\mathbb R^{d}$
the spectral resolution of 
the Laplacian $\Delta = - \partial^2_1 -
\cdots - \partial_d^2$
may be given in terms of convolutions
with the Fourier transform of
the measures $d \sigma_{r}$, induced on the spheres 
centered at the origin by the Lebesgue
measure, since
$\Delta f * \widehat {\sigma_{r}}
= r^{2} f * \widehat {\sigma_{r}}$.
The celebrated
Stein-Tomas theorem
\cite[Ch.~9]{Stein}
describes the mapping properties
of 
the convolution operator 
with $\widehat{\sigma_{r}}$. Throughout the paper, we adopt the following notation
\begin{equation}\label{punto-critico}
p_*(d):= 2 \frac{d+1}{d+3}\,, \,\,d\in\NN\,.
\end{equation}
\begin{theorem} [Stein-Tomas
Restriction Inequality]
Suppose that $1 \leq p \leq p_*(d)$ and let
$\frac 1 p + \frac 1{p'} = 1$.
Then the estimate
\begin{equation}\label{PeterTomas}
\| f * \widehat{\sigma_{r}}\|_{p'}
\leq C_{r} \| f \|_{p}
\end{equation}
holds for
all Schwartz functions on $\mathbb R^{d}$
and all $r >0$.
\end{theorem}

According to the Knapp
example \cite{Stein},  estimate \eqref{PeterTomas}
fails  if $p>p_{*}
 $.

Strichartz suggested
to study the boundedness properties
of the operators arising in the spectral resolution
of other Laplacians. 
D. M\"uller,
motivated also by the  works
of C. Sogge on the spectral projections of the
Laplace-Beltrami operator
on compact Riemannian
manifolds \cite{So1}, \cite{So2}, 
proved
an analogue 
of the Stein-Tomas theorem
for the sublaplacian on the  Heisenberg group
\cite{Mu}.

\begin{theorem}
\label{annals}[D. M\"uller]
Let $1 \leq p \leq 2$
and  $\frac 1 p + \frac 1{p'} = 1$.
The inequality
\begin{equation}
\label{teorema di muller}
\| \mathcal P^L_{\mu} f\|_{L^{\infty}_{t}L^{p'}_{z}}
\leq
C_{\mu}
\|f\|_{L^{1}_{t}L^{p}_{z}},
\end{equation}
holds
for all Schwartz functions on $\mathbb H^n$
and all $\rho >0$.
\end{theorem}

The theorem is stated in terms of the mixed
Lebesgue norms
\begin{equation}
\label{norme di muller}
\|f\|_{L^{r}_{t}L^{p}_{z}}
=
\left(
\int\limits_{\mathbb C^{n}}
\left(
\int\limits_{-\infty}^{\infty}
| f(z,t)|^{r}
dt
\right)^{\frac p r} dz
\right)^{\frac 1p},
\qquad
1 \leq p,r < \infty,
\end{equation}
(with the obvious modifications
when $p$ or $r$
are equal to $\infty$),
since, as shown by M\"uller, the only
available estimate between
$L^p$ spaces on $\mathbb H^n$ is the trivial
$L^{1}-L^{\infty}$ one.
In addition, a
counterexample produced
in \cite{Mu} shows that $ \mathcal P^L_{\mu}$ is unbounded as an operator between 
${L^{r}_{t}L^{p}_{z}}$
and ${L^{r'}_{t}L^{p'}_{z}}$,
unless $r=1$.
The main reason for that is that
the operators $\mathcal P^L_{\mu}$
operate on the $t$ variable
through the Fourier transform,
but the Heisenberg group has
one dimensional centre and there
are no nontrivial restriction estimates
for the Fourier transform on
the real line.
Indeed, S. Thangavelu proved in \cite{Th} that the inequality
$$
 \| \mathcal P^L_{\mu} f\|_{L^{p'}(G)}
 \leq
C \|  f\|_{L^{p}(G)}
$$
holds for $1 \leq p \leq p_*(n)$
on the direct product $G$ of $n$
copies of the
three dimensional Heisenberg group $\mathbb H_{1}$.

We extend 
M\"uller's theorem in two ways.
First, since the dimension of the centre of  
 M\'etivier  groups is
in general 
bigger than one
(actually, in the M\'etivier class only the Heisenberg groups have a one dimensional centre),
we incorporate
the Stein-Tomas theorem
in the estimate concerning
the central variables.
Secondly, we  
improve \eqref{teorema di muller}
by replacing on the left-hand side $p'$
with an exponent $q< p'$.
More precisely, we will prove the following
result.

\begin{theorem}
\label{nostro}
Let $G$ be a M\'etivier  group,
with Lie algebra $\mathfrak{g}$.
Let $\mathfrak{z}$ and
 $\mathfrak{v}$ 
 denote, respectively, the centre of $\mathfrak{g}$ and 
its orthogonal complement.
\\
If
$\dim \mathfrak{z}=d$ and $\dim \mathfrak{v}=2n$,
and if  $1\le r\leq p_{*}(d) 
%= 2\frac{d+1}{d+3}
$, 
then for all $p, q$ satisfying $1\leq p\leq 2 \leq q \leq \infty$ and
for all Schwartz functions
$f$, we have
\begin{equation}\label{secondainterpolata 1}
\|\mathcal{P}^L_\mu  f\|_{L^{r'}(\mathfrak{z}) L^q(\mathfrak{v}) }
\le C
\mu^{d(\frac{2}{r}-1)+n (\frac{1}{p}-\frac1q)-1 }
\|f\|_{ L^r (\mathfrak{z})  L^p (\mathfrak{v})  }\,,\,\,\mu>0\,.
\end{equation}
\end{theorem}
Here $C$ is independent of $f$ and
$\mu$ and the definition of the norms
is analogous to that given  in the Heisenberg framework.

To explain the strategy, we recall
that
the operators $\mathcal P^L_{\mu}$
are given by
 the action in $z$ of the spectral projections of the twisted Laplacian, conjugated with
the Fourier transform in the central
variable.
The twisted Laplacian is a second order elliptic differential operator
 on $\mathbb C^{n}$ with
point spectrum.
The estimates of the norms
between Lebesgue spaces of
its spectral projections, which are
 an essential ingredient %tool
in the proof of \eqref{secondainterpolata 1},
 have been progressively strengthened
in the last twenty years
(see for instance \cite{RRTh}, \cite{SZ});
eventually the sharp $L^{p}-L^{2}$
bounds have been attained by H. Koch and F. Ricci \cite{KR}
(see also \cite{CaCia} for a different proof).
By incorporating in our
 argument the optimal
estimates
we obtain a result,
which improves on  that of
M\"uller 
also on the Heisenberg group (in fact,
for a Schwartz function $f$ on $\mathbb H^{n}$ we prove that
\begin{equation*}
\| \mathcal P^L_{\mu} f\|_{L^{\infty}_{t}L^{2}_{z}}
\leq
C_{\mu}
\|f\|_{L^{1}_{t}L^{p}_{z}},
\end{equation*}
for all $1 \leq p \leq 2$).

\begin{comment}
More generally, on a M\'etivier group $G$
of dimension $d+2n$,
with centre of dimension $d$,
we prove that for all Schwartz functions
$f$ and all $\mu>0$
\begin{equation}
\label{totale}
\|\mathcal{P}^L_\mu  f\|_{L^{r'}(\mathfrak{z})L^q(\mathfrak{v}) }
\leq C
\mu^{d(\frac{2}{r}-1)+n (\frac{1}{p}-\frac1q)-1 }
\|f\|_{L^r (\mathfrak{z})L^p (\mathfrak{v})}\,
\end{equation}
when $1 \leq r \leq 2 \frac{d+1}{d+3}$
and $1 \leq p \leq 2 \leq q \leq \infty$,
\end{comment}

It is worth noticing
that from the estimates for the operators
$\mathcal{P}^L_\mu$ one can deduce 
estimates for
the standard spectral projections
of the sublaplacian, which could be  used to prove $L^p$
summability results for Bochner-Riesz
means associated to the sublaplacian
(see \cite{Mueller2} for the Heisenberg case).
We shall address this problem
in the framework of M\'etivier groups in
a forthcoming paper \cite{CaCia2}.

\medskip

The paper is organized as
follows.
In Section 2 we
 recall the
spectral resolution
of the sublaplacian
on the Heisenberg group,
and 
use
the Koch-Ricci
estimates for the twisted laplacian
to strengthen M\"uller's estimate.
 In Section 3 we 
present some restriction estimates for the full Laplacian (defined by \eqref{laplaciano-full-completo}) on 
the Heisenberg group.
In Section 4
we compute
the spectral resolution of the sublaplacian
on a M\'etivier group $G$.
Following a well known procedure
(see \cite{Str} and \cite{Tay}),
by taking the Radon transform
in the central variables
and using the M\'etivier property,
we reduce the computation
of the spectral decompositon
of a function on $G$ to the
spectral decomposition
of its Radon transform
 on a Heisenberg group.
In Section 5
we prove the restriction theorem on
the M\'etivier groups.
\begin{comment}
There are two players in this game:
the Stein-Tomas restriction theorem
applied to the Fourier
transform in the central variables
and the
Koch-Ricci estimates
for the spectral projections of the twisted
Laplacian, 
in the remaining variables.
% on the complement of the centre.
\end{comment}
An essential tool is given by 
a conditional statement,
based on the assumption that the spectral
projections  of the twisted
Laplacian
are bounded between
two Lebesgue spaces.
\begin{comment}
From this result we obtain the main theorem
by taking into account
the estimates for the norms
of the projections of the twisted
Laplacian.
\end{comment}
We conclude by showing that
the range of $r$ in \eqref{secondainterpolata 1} is
sharp.

%%%%%%%%%%%%%%%%%%%%

%%%%%%%%%%%%%

\section{Restriction estimates
for the sublaplacian on the Heisenberg group}

In this section we discuss the case of the Heisenberg group.
The Heisenberg group $\HH_{n}$
is the space $\RR^{n} \times \RR^{n}\times \RR$
% of the triplets  $(x, y, t) = (x_{1}, \cdots x_{n}, y_{1}, \cdots y_{n}, t)$,
 equipped
with the product
\begin{equation*}
(x, y,t)(x', y',t') = \left(x+x',
y+y', t+t' + \frac 12 (x \cdot y'
-x' \cdot y)\right),
\end{equation*}
for $x, x', y, y'$ in $\RR^{n}$
and $t, t'$ in $\RR$.
%$(x,y,t), (x',y',t')$ in $\HH_{n}$.
This product turns $\HH_{n}$ into a two step nilpotent Lie group with centre
given by $\{ (0,0,t): t \in \RR \}$.

The algebra, $\mathfrak h_{n}$, of left invariant vector fields
on  $\HH_{n}$ is spanned by
\begin{equation*}
X_{j} = \frac {\partial}{\partial x_{j}}
- \frac 12 y_{j }\frac {\partial}{\partial t},
\quad
Y_{j} = \frac {\partial}{\partial y_{j}}
+ \frac 12 x_{j} \frac {\partial}{\partial t},
\quad
T = \frac {\partial}{\partial t},
\end{equation*}
$j = 1, \cdots, n$.
In terms of these vector fields we
introduce on $\HH_{n}$ the
sublaplacian
\begin{equation*}
L = - \sum_{j=1}^{n}
\left(
X_{j}^{2}+Y_{j}^{2}
\right),
\end{equation*}
which is hypoelliptic
since the set $\{X_{1} \cdots, Y_{n}\}$
generates $\mathfrak h_{n}$ as a Lie algebra,
and the full Laplacian
\begin{equation}\label{laplaciano-full-completo}
\Delta_{\HH} = - \sum_{j=1}^{n}
\left(
X_{j}^{2}+Y_{j}^{2}
\right) - T^{2} = L - T^{2}.
\end{equation}

We will use complex coordinates
\begin{equation*}
z_{j} = x_{j} + i y_{j}
\quad
\text{and}
\quad
\overline z_{j} = x_{j} - i y_{j},
\end{equation*}
$j = 1, \cdots, n$.
In these coordinates the Haar
measure coincides with the Lebesgue measure $dz dt = dx dy dt$.

\medskip

The operators $L$ and $- i T$
%, or equivalently the operators $L$ and $\Delta_{\HH}$,
extend to a pair of strongly
commuting self-adjoint operators on
$L^2(\HH^n)$.
They therefore
admit a joint spectral decomposition,
that we now briefly recall for the sake of completeness.
For
more details
we refer the reader to 
the book \cite{Th2}.

Given a nonzero real number
$\lambda$ and a point $(z,t) = (x+iy,t)$ in $\HH_{n}$, we denote by $\pi_{\lambda}(z,t)$
the operator acting on $L^{2}(\RR^{n})$ defined by
\begin{equation*}
\pi_{\lambda}(z,t) \phi (\xi)
= e^{i \lambda t}
\pi_{\lambda}(z) \phi (\xi)
= e^{i \lambda \left(t + x \cdot \xi + \frac 12 x \cdot y \right)} \phi(\xi + y),
\end{equation*}
where $\pi_{\lambda}(z)
= \pi_{\lambda}(z, 0)$, so that
$\pi_{\lambda}(z,t)= e^{i \lambda t} \pi_{\lambda}(z,0)= e^{i \lambda t} \pi_{\lambda}(z)$.
For each $\lambda \neq 0$ the map $\pi_{\lambda}$ %,
 from $\HH_{n}$ to the group of unitary operators on $L^{2}(\RR^{n})$ %,
 %furnishes 
 is an irreducible
representation of $\HH_{n}$.
These maps
are called 
Schr\"odinger's  representations.
In terms of them
we define 
the group Fourier transform
of a Schwartz function $f$ on $\HH_{n}$,
which is given by
%The group Fourier transform of a Schwartz function $f$ on
%$\HH_{n}$ is the function mapping $\lambda$ to the operator
\begin{align*}
\RR \setminus \{0\} \ni \lambda \mapsto \pi_{\lambda}(f)
&= \int\limits_{\HH_{n}}
f(z,t) \pi_{\lambda}(z,t) dt dz
= \int\limits_{\CC_{n}}
f^{(\lambda)} (z) \pi_{\lambda}(z)
dz,
\end{align*}
where 
$$
f^{(\lambda)} (z)
= \int\limits_{-\infty}^{\infty}
f(z,t) e^{i \lambda t} dt
$$
is the ordinary Fourier transform 
of $f(z,t)$ in
the central variable $t$.
%The Schr\"odinger's representations suffice to reconstruct a Schwartz function from its Fourier transform, 
The information provided by
$\pi_{\lambda}(f)$
suffices to reconstruct $f$, which is
in fact given by
\begin{align}
\label{formula di inversione}
f(z,t) 
%&= \int\limits_{-\infty}^{\infty}
%\text{tr} \left(  \pi_{\lambda}(z,t)^{*}\pi_{\lambda}(f)\right) |\lambda|^{n} d \lambda\\
&= \frac{1}{(2\pi)^{n+1}} \int\limits_{-\infty}^{\infty}
e^{- i \lambda t}\,
\text{tr} \left(  \pi_{\lambda}(z)^{*}\pi_{\lambda}(f)
\right) |\lambda|^{n} d \lambda,
\end{align}
where we denote by
$\pi_{\lambda}(z)^{*}$
the adjoint of $\pi_{\lambda}(z)$
and  by $\text{tr}$
the trace of an operator on $L^{2}(\RR^{n})$.

The differential $d \pi_{\lambda}$ of $\pi_{\lambda}$
yields a representation of the
Lie algebra $\mathfrak 
h_{n}$.
This representation extends
to a representation of the
universal enveloping algebra %, 
%which we shall
 denoted by the same symbol
  $d \pi_{\lambda}$.
%1%1
The derivative with respect to the central
variable, $T$, 
is represented in this picture by
the multiplication by $i \lambda$.
The sublaplacian $L$ is represented by the rescaled Hermite operator
\begin{equation*}
d \pi_{\lambda}(L) = \Delta + \lambda^{2} |\xi|^{2},
\end{equation*}
where $\Delta
= -\partial_1^2 - \cdots - \partial^2_n$ and $|\xi|^{2} = \xi_{1}^{2} +
\cdots + \xi_{n}^{2}$.

The Hermite operator $d \pi_{\lambda}(L)$ has a pure point spectrum with eigenvalues $|\lambda| (2k + n)$ for
$k = 0, 1, \cdots$. 
The eigenspace corresponding to the eigenvalue
$|\lambda| (2k + n)$
has an orthonormal basis given by
\begin{equation*}
\{ \Phi^{(\lambda)}_{\alpha}: |\alpha| = k \},
\end{equation*}
where the functions
$\Phi^{(\lambda)}_{\alpha}$
are defined, for each multiindex $\alpha = (\alpha_1,
\cdots, \alpha_n)$ in $\ZZ_+^n$ of length
$|\alpha| = \alpha_1 + \cdots+ \alpha_n = k$, by
\begin{equation*}
\Phi^{(\lambda)}_{\alpha} (\xi_1, \cdots, \xi_n) =
|\lambda|^{\frac n4}
h_{\alpha_{1}}(\sqrt{|\lambda|} \xi_1) \cdots h_{\alpha_{n}}(\sqrt{|\lambda|} \xi_n)
\end{equation*}
and the functions $h_i(t)$ are
normalized
one-dimensional Hermite functions
 %with norm one
 in $L^2(\RR, dt)$.

Therefore, $d \pi_{\lambda}(L)$ is 
represented by
\begin{equation}\label{dpi di L}
d \pi_{\lambda}(L) = \sum_{k=0}^{\infty}
|\lambda| (2k + n) P^{\lambda}_{k},
\end{equation}
where $P^{\lambda}_{k}: L^{2}(\RR^{n}) \rightarrow L^{2}(\RR^{n})$
is the projection onto the eigenspace
corresponding to $|\lambda| (2k + n)$.

Inserting in 
 \eqref{formula di inversione}
 the decomposition 
$ \sum_{k=0}^{\infty} P^{\lambda}_{k}$ of
 the identity
 operator on $L^{2}(\RR^{n})$,
we obtain
\begin{align}\label{formula di inversione 2}
f(z,t) &= \frac{1}{(2\pi)^{n+1}} \sum_{k=0}^{\infty}
\int\limits_{-\infty}^{\infty}
e^{-i \lambda t}\,
\text{tr} \left( \pi_{\lambda}(z)^{*}
\pi_{\lambda}(f) P^{\lambda}_{k} \right) |\lambda|^{n} d \lambda.
\end{align}
This decomposition may be thought of as
the expansion of $f$
in joint eigenfunctions,
$e^{-i \lambda t}\,
\text{tr} \left( \pi_{\lambda}(z)^{*}
\pi_{\lambda}(f) P^{\lambda}_{k} \right)$, of $-i T$ and $L$.

\medskip

To obtain a more explicit form for
\eqref{formula di inversione 2} 
%this expansion
 we compute the trace of the operators
$\pi_{\lambda}(z)^{*} \pi_{\lambda}(f) 
P^{\lambda}_k$,
which is given by
\begin{equation*}
\text{tr} \left( \pi_{\lambda}(z)^{*}
\pi_{\lambda}(f) 
P^{\lambda}_k
 \right) =
% e^{-i{\lambda}t}
 \sum_{|\alpha| = k} \left( \Phi^{\lambda}_\alpha, \pi_{\lambda}(z)^{*}
\pi_{\lambda}(f) 
\Phi^{\lambda}_\alpha
\right)_{L^2(\RR^n)}.
\end{equation*}
The sum may be expressed in
a close form,
introducing the $\lambda$-twisted convolution
on $\CC^{n}$.

\begin{definition}
Let $\lambda$ be a nonzero real number.
The 
{\rm{$\lambda$-twisted convolution}}
$h\times_\lambda g$ of two Schwartz functions $h,g$ on
$\CC^n$
%$L^1 (\CC^n, dz)$
is defined by
\begin{equation}
\label{convoluzione twisted su H_n}
(h\times_\lambda g)(z)=
\int_{\CC^n}
h(z-w)g(w)e^{\frac{i}{2} \lambda \Im m z\cdot  {\bar{w}}}
dw\,.
\end{equation}
When $\lambda = 1$ we shall write
$\times$ instead of
$\times_{1}$.
\end{definition}

Then
\begin{align*}%\label{traccia 1}
%\text{tr} \left( \pi_{\lambda}(z)^{*}\pi_{\lambda}(f) P^{\lambda}_k \right) &=
% e^{-i{\lambda}t}
 \sum_{|\alpha| = k} \left( \Phi^{\lambda}_\alpha, \pi_{\lambda}(z)^{*}
\pi_{\lambda}(f) 
\Phi^{\lambda}_\alpha
\right)_{L^2(\RR^n)}
&= 
\frac 1 {(2\pi)^n}
%e^{-i {\lambda} t}
f^{({
{\lambda}
})} 
\times_{{\lambda}}
\varphi^{{\lambda}}_{k}(z)
\end{align*}
(see \cite{Th2}),
where 
$$
\varphi_k^\lambda (z) =
\sum_{|\alpha| = k} (\pi_{\lambda}(z) \Phi_{\alpha}^{(\lambda)},
\Phi_{\alpha}^{(\lambda)} )_{L^2(\RR^n)}
= \tilde \varphi_k (\sqrt {|\lambda|} |z|),
$$
and
$\tilde \varphi_k$ is
the $k$th-Laguerre function normalized by $\| \tilde \varphi_{k}\|_{L^{2}(\RR_{+}, t^{n-1}dt)} = 1$.
If $g$ is a Schwartz function on
$\CC^{n}$ and $\lambda \neq 0$,
we set
\begin{equation}
\label{Lambda 1}
\Lambda^{\lambda}_{k} g (z) =
\frac 1 {(2\pi)^n}
g \times_{\lambda}
\varphi^{|\lambda|}_{k}(z),
\end{equation}
writing %\eqref{traccia 1} becomes
\begin{align*}%\label{traccia 2}
\text{tr} \left( \pi_{\lambda}(z)^{*}
\pi_{\lambda}(f) 
P^{\lambda}_k
 \right) &= 
%e^{-i {\lambda} t}
\Lambda_{k}^{
{\lambda}} 
f^{({\lambda})}(z).
\end{align*}
%It is easy to see that t
The operators $\Lambda_{k}^{\lambda}$
are orthogonal projections in $L^2(\CC^n, |\lambda|^{n} dz)$, since
\begin{equation*}
\varphi_j \times \varphi_j = (2\pi)^n \varphi_j
\quad
\text{and}
\quad
\varphi_j \times \varphi_k = 0
\quad
\text{if $j \neq k$}
\end{equation*}
(see \cite[(1.4.30)]{Th2}), which imply
\begin{equation*}
|\lambda|^{2n}
\varphi_j^{|\lambda|} \times \varphi^{|\lambda|}_k = (2\pi)^n
\delta_{jk} 
|\lambda|^{n}
\varphi^{|\lambda|}_j.
\end{equation*}

\medskip

Then \eqref{formula di inversione 2} takes the form
\begin{align*}
%\label{formula di inversione 3}
f(z,t) &= \frac 1 {(2\pi)^{n+1}}
\sum_{k=0}^{\infty}
\int\limits_{-\infty}^{\infty}
e^{-i \lambda t}\,
\Lambda_{k}^{
{\lambda}} 
f^{({\lambda})}(z)
 |\lambda|^{n} d \lambda.
\end{align*}
This decomposition
 together with the allied Plancherel formula
\begin{align}\label{Plancherel}
\int\limits_{\CC^{n}}
\int\limits_{-\infty}^{\infty}
|f(z,t)|^{2} dt dz
&= \frac 1 {(2\pi)^{2n+1}}
\sum_{k=0}^{\infty}
\int\limits_{-\infty}^{\infty}
\left\|
\Lambda_{k}^{
{\lambda}} 
f^{({\lambda})}
\right\|_{L^2(\mathbb C^n)}^{2} |\lambda|^{2n} d \lambda,
\end{align}
is the starting point 
for the development of the joint functional calculus of $L$ and $T$. 
{Indeed,
given a bounded function $m : \RR_{+}
\times \RR\setminus \{0\}
\rightarrow \CC$, we define
for  a Schwartz function $f$
\begin{align}\label{definizione di m}
m(L, -iT) f (z,t) &= \frac 1 {(2\pi)^{n+1}}
\sum_{k=0}^{\infty}
\int\limits_{-\infty}^{\infty}
m(|\lambda| (2k+n), \lambda)
e^{-i \lambda t}\,
\Lambda_{k}^{{\lambda}} 
f^{({\lambda})}(z)
 |\lambda|^{n} d \lambda.
\end{align}
Then by \eqref{Plancherel} we have
\begin{align*}
\int\limits_{\CC^{n}}
\int\limits_{-\infty}^{\infty}
|m(L, -iT) f(z,t)|^{2} dt dz
&\leq \|m\|^2_{L^{\infty}(\RR_{+}
\times \RR\setminus \{0\})}
\int\limits_{\CC^{n}}
\int\limits_{-\infty}^{\infty}
|f(z,t)|^{2} dt dz.
\end{align*}}

We shall use \eqref{definizione di m} to
introduce the operators $\delta_{\mu} (L)$
and $\delta_{\mu}(\Delta_{\HH})$
for $\mu > 0$,
which are defined for
a Schwartz function $f$ by
\begin{equation*}
\delta_{\mu} (D) f = \lim_{\epsilon \rightarrow 0+}
\frac 1 {2 \epsilon} \chi_{(\mu - \epsilon,
\mu + \epsilon)} (D) f,
\end{equation*}
with $D = L, \Delta_{\HH}$,
where $\chi_{(\mu - \epsilon,
\mu + \epsilon)}$ is the characteristic function
of the interval ${(\mu - \epsilon,
\mu + \epsilon)}$.

More generally,
with the same techniques one can also consider operators of the form
 $\delta_{\mu}\left( m(L, -iT)
\right)$ for a suitable
 function $m$.
 In the following, $\lambda$ refers to the spectrum of $-iT$,
 and $|\lambda|(2k+n)$ to the spectrum of $L$.
If $m$ is an even function of $\lambda$,
we may rewrite
\eqref{definizione di m} as
\begin{align*}%\label{definizione di m}
m(L, -iT) f (z,t) &= \frac 1{(2\pi)^{n+1}}
\sum_{k=0}^{\infty}
\int\limits_{0}^{\infty}
m(\lambda (2k+n), \lambda)
\Big(
e^{-i \lambda t}\,
\Lambda_{k}^{{\lambda}} 
f^{({\lambda})}(z)
+
e^{i \lambda t}\,
\Lambda_{k}^{{-\lambda}} 
f^{({-\lambda})}(z)
\Big)
\lambda^{n} d \lambda.
\end{align*}
We also assume that
 $m(\lambda (2k+n), \lambda)$ is %, for each $k\in\ZZ_{+}$, a 
 %strictly monotone increasing,
a differentiable function of
$\lambda$ on $\RR_{+}$, with strictly positive derivative, satisfying
$\lim_{\lambda \rightarrow 0+} 
m(\lambda (2k+n), \lambda) = 0$
and $\lim_{\lambda \rightarrow + \infty} 
m(\lambda (2k+n), \lambda) = + \infty$.
Then
the equation $m(\lambda (2k+n), \lambda) = \mu$ may be solved
for each $k$
to give $\lambda = \lambda^m_{k}(\mu)$.
For notational simplicity, we shall write $\lambda_{k}(\mu)$ instead of $\lambda^m_{k}(\mu)$
and denote by $\lambda'_{k}$  the
derivative of $\lambda_{k}$.
\\
Replacing in the integral $\lambda$
with $\mu$, we obtain
\begin{align*}%\label{definizione di m}
m(L, -iT) f (z,t) &= \frac 1{(2\pi)^{n+1}}
\sum_{k=0}^{\infty}
\int\limits_{0}^{\infty}
\mu
\Big(
e^{-i \lambda_{k}(\mu) t}\,
\Lambda_{k}^{{\lambda_{k}(\mu)}} 
f^{({\lambda_{k}(\mu)})}(z)
\\
&\qquad\qquad+
e^{i \lambda_{k}(\mu) t}\,
\Lambda_{k}^{{-\lambda_{k}(\mu)}} 
f^{({-\lambda_{k}(\mu)})}(z)
\Big)
 \lambda_{k}(\mu)^{n} 
\lambda'_{k}(\mu)
% {{d  \lambda_{k}(\mu)}\over{d \mu}}
  d \mu,
\end{align*}
which is the spectral decomposition of $m(L,-iT)$.
Hence, the spectral resolution with respect to $m(L,-iT)$ of a Schwartz function $f$ is
\begin{align}\label{decomposizione di f}
\notag
f (z,t) &= \frac 1{(2\pi)^{n+1}}
\sum_{k=0}^{\infty}
\int\limits_{0}^{\infty}
\Big(
e^{-i \lambda_{k}(\mu) t}\,
\Lambda_{k}^{{\lambda_{k}(\mu)}} 
f^{({\lambda_{k}(\mu)})}(z)
\\
&\qquad\qquad+
e^{i \lambda_{k}(\mu) t}\,
\Lambda_{k}^{{-\lambda_{k}(\mu)}} 
f^{({-\lambda_{k}(\mu)})}(z)
\Big)
 \lambda_{k}(\mu)^{n} 
\lambda'_{k}(\mu)
% {{d  \lambda_{k}(\mu)}\over{d \mu}}
  d \mu.
\end{align}

%2%

Given a Schwartz function $f$, its spectral resolution 
is given in terms of the distributions
$$
{\mathcal P}^{m}_{\mu} f =
\delta_{\mu}(m(L,-iT))f = \lim_{\varepsilon \rightarrow 0+} \frac 1{2 \varepsilon} \chi_{({\mu}
-\varepsilon,
{\mu}
+\varepsilon)}(m(L,-iT))f.
$$
Since $\text{tr} \left( \pi_{\mu}(z,t)^{*}
\pi_{\mu}(f)\right)$ is a continuous function
of $\mu$, this limit
%in \eqref{decomposizione su H-n} 
exists and is given by
%Plugging \eqref{traccia 2} in this formula we obtain
\begin{equation}
\label{Proiettore di m}
{\mathcal P}^{m}_{\mu} f(z,t)
=
\sum_{k=0}^{\infty}
 \frac {\lambda_{k}(\mu)^{n}\lambda'_{k}(\mu)}{(2\pi)^{n+1}} 
 %{{d  \lambda_{k}}\over d \mu}(\eta)
\Big(
e^{-i  \lambda_{k}(\mu) t}\,
\Lambda_{k}^{
{\lambda_{k}(\mu)}} 
f^{{(\lambda_{k}(\mu))}}(z)
+
e^{i \lambda_{k}(\mu) t}\,
\Lambda_{k}^{
{-\lambda_{k}(\mu)}} 
f^{{(-\lambda_{k}(\mu))}}(z)
\Big) .
\end{equation}
The inversion formula
\eqref{decomposizione di f}
may be written
%in terms of n terms of the ditributions
%${\mathcal P}^{m}_{\mu} f$
\begin{align*}
f(z,t) &=
\int\limits_{0}^{\infty}
{\mathcal P}^{m}_{\mu} f (z,t) d \mu,
\end{align*}
where the integral converges in the
sense of distributions.

In the case of the sublaplacian $m(\mu, \lambda) =
\mu$, thus we have
$\mu = |\lambda|(2k+n)$,
which yields
$
\lambda_{k}(\mu) = \mu/ {2k + n}.
$
For notational simplicity, we shall write $\mu_k$ instead of
$\lambda_k (\mu)$.

Therefore, 
\begin{align}
\label{Plambda su Hn}
{\mathcal P}^{L}_{\mu} f(z,t) &=
 \frac {\mu^{n}}{(2\pi)^{n+1}} 
\sum_{k=0}^{\infty}
 \frac 1{(2k+n)^{n+1}} 
\Big(
e^{-i  \mu_{k} t}\,
\Lambda_{k}^{
{\mu_{k}}} 
f^{{(\mu_{k})}}(z)
+
e^{i (\mu_{k}) t}\,
\Lambda_{k}^{
{-\mu_{k}}} 
f^{{(-\mu_{k})}}(z)
\Big) .
\end{align}

In the case of the full Laplacian
$m(\mu, \lambda) = \mu + \lambda^{2}$,
hence
$\mu = |\lambda|(2k+n)+\lambda^{2}$
and%%
\begin{equation}\label{lambda per il
laplaciano completo}
\lambda_{k}(\mu)
= \frac12 \sqrt{4 \mu + (2k+n)^{2}}
- \frac{2k+n}2.
\end{equation}
Therefore,
\begin{align*}
{\mathcal P}^{\Delta_{\HH}}_{\mu} f(z,t) &=
\frac 1{4^{n} \pi^{n+1}} 
\sum_{k=0}^{\infty}
\frac{\left(
\sqrt{4 \mu + (2k+n)^{2}}
- 2k-n
\right)^{n}}{\sqrt{4 \mu + (2k+n)^{2}}}
\times
\\
&
\times\Big(
e^{-i  \lambda_{k}(\mu) t}\,
\Lambda_{k}^{
{\lambda_{k}(\mu)}} 
f^{{(\lambda_{k}(\mu))}}(z)
+
e^{i \lambda_{k}(\mu) t}\,
\Lambda_{k}^{
{-\lambda_{k}(\mu)}} 
f^{{(-\lambda_{k}(\mu))}}(z)
\Big) .
\end{align*}

\medskip

The operators $\Lambda_{k}^{\lambda}$,
defined by \eqref{Lambda 1}, are
the spectral projection of the %, so called,
twisted Laplacian on $\CC^{n}$, that we now introduce.
Given a Schwartz function $f$,
consider the Fourier transform in the central variable of
the functions $X_{j}f$ and $Y_{j}f$.
Integrating by parts we obtain
\begin{align*}
(X_{j} f)^{(\lambda)}(z)
&= \left(
\frac {\partial}{\partial x_{j}}
+ \frac i2 \lambda y_{j }
\right)
f^{(\lambda)}(z),
\end{align*}
and similarly for $Y_{j}$.
Hence, setting
\begin{equation*}
X^{(\lambda)}_{j} = \frac {\partial}{\partial x_{j}}
+ \frac i2 \lambda y_{j }
\quad
\text{and}
\quad
Y^{(\lambda)}_{j} = \frac {\partial}{\partial x_{j}}
- \frac i2 \lambda y_{j },
\end{equation*}
we have
\begin{equation*}
(X_{j} f)^{(\lambda)}
= X^{(\lambda)}_{j} f^{(\lambda)}
\quad
\text{and}
\quad
(Y_{j} f)^{(\lambda)}
= Y^{(\lambda)}_{j} f^{(\lambda)}.
\end{equation*}
These formulas imply
\begin{equation*}
(Lf)^{(\lambda)} = - \sum_{j=1}^{n}
\left(
(X^{(\lambda)}_{j})^{2}
+ (Y^{(\lambda)}_{j})^{2}
\right) f^{(\lambda)}
= \Delta^{(\lambda)} f^{(\lambda)},
\end{equation*}
where $\Delta^{(\lambda)}$, for $\lambda \neq 0$, is the
{\it $\lambda$-twisted Laplacian}.
Note that $\Delta^{(0)}$ is the Laplacian on $\CC^n$.

From properties of the twisted convolution and of the Laguerre functions
(see \cite{ Th2}),
it follows that for any Schwartz function $g$ on $\CC^{n}$
\begin{equation*}
\Delta^{(\lambda)}
\left(
\Lambda_{k}^{\lambda} g
\right )
=
|\lambda| (2k +n) \Lambda_{k}^{\lambda} g
\end{equation*}
and
\begin{equation*}
g = \sum_{k=0}^{\infty}
\Lambda_{k}^{\lambda } g.
\end{equation*}
Therefore, for $\lambda \neq 0$ the operators
$\Lambda_{k}^{\lambda}$ are the spectral projections
associated to $\Delta^{(\lambda)}$.
When $\lambda = 1$,
we will write for simplicity $\Lambda_{k}$,
instead of $\Lambda^{1}_{k}$.

\medskip

We now prove estimate \eqref{secondainterpolata 1} in
Theorem \ref{nostro} in the special case of the Heisenberg group, thus
improving on \eqref{teorema di muller} in Theorem \ref{annals}.
%We shall also prove a similar bound
%for $\mathcal P^{\Delta_{\HH}}_{\mu} f$.
%We start by briefly reviewing M\"uller's argument, extending its
%scope to include operators associated to more general functions of $L$ and $T$.
The first step is a simple lemma that easily follows 
from dilation arguments.

\begin{lemma}\label{lemmanormaPirok}
Suppose that $\Lambda_k :L^p (\CC^n)\to L^q (\CC^n)$
for some $p,q$.
If $g$ lies in $\cS (\CC^n)$,
then for all $\lambda > 0$ we have
\begin{equation*}
\big\|
\Lambda^\lambda_k g \big\|_{L^q(\CC^n)} 
%=\big\|g\times_\mu \varphi^\mu_k\big\|_{L^q (\CC^n)}
 \leq
\lambda^{n(\frac1p-\frac1q-1)}
\big\| \Lambda_k \big\|_{L^p(\CC^n)
 \rightarrow
 L^q(\CC^n)
}
\|g\|_{L^p(\CC^n)}.
\end{equation*}
\end{lemma}

\begin{remark}
Observe that $\big\| \Lambda^{-1}_k \big\|_{L^p(\CC^n)
 \rightarrow
 L^q(\CC^n)
} = \big\| \Lambda^{1}_k \big\|_{L^p(\CC^n)
 \rightarrow
 L^q(\CC^n)
} =\big\| \Lambda_k \big\|_{L^p(\CC^n)
 \rightarrow
 L^q(\CC^n)
}
$.
\end{remark}

Then we prove the following
conditional statement.

\begin{proposition}\label{stima astratta} Let $\mu > 0$.
If $\Lambda_{k}$ is bounded from
$L^{p}(\CC^{n})$ to $L^{q}(\CC^{n})$, 
then
\begin{align}\label{stima astratta 11}
\left
\|\mathcal P^{m}_{\mu} f
\right\|_{L^{\infty}_{t}L^{q}_{z}}
&\leq
\frac {2 \| f \|_{L^{1}_{t}L^{p}_{z}}}
{(2\pi)^{n+1}}
\left(
\sum_{k=0}^{\infty}
{(\lambda_{k}(\mu))^{
n\left(\frac1p - \frac1q
\right)}
\lambda'_{k}(\mu)}
\|\Lambda_{k}
\|_{L^{p}(\CC^{n}) \rightarrow L^{q}(\CC^{n})}
\right)
\end{align}
for all Schwartz functions $f$ on $\HH_{n}$.
\end{proposition}

\begin{proof}
To simplify the notations we
write 
$f$ as if it were the product of two functions, 
that is $f(z,t) = h(t) g(z)$.
Then \eqref{Proiettore di m} becomes
\begin{align*}
{\mathcal P}^{m}_{\mu} f(z,t) &=
\sum_{k=0}^{\infty}
 \frac {(\lambda_{k}(\mu))^{n}\lambda'_{k}(\mu)}{(2\pi)^{n+1}} 
\Big(
e^{-i  (\lambda_{k}(\mu)) t}\,
\hat h ((\lambda_{k}(\mu)))
\Lambda_{k}^{
{(\lambda_{k}(\mu))}} 
g(z)
+
e^{i (\lambda_{k}(\mu)) t}\,
\hat h (-(\lambda_{k}(\mu)))
\Lambda_{k}^{
{-(\lambda_{k}(\mu))}} 
g(z)
\Big) .
\end{align*}
Since $|\hat h (\lambda)| \leq \| h\|_{L^{1}(\RR)}$ for all $\lambda$, we obtain
\begin{align*}
\left| {\mathcal P}^{m}_{\mu} f(z,t) \right |
&\leq
\|h\|_{L^{1}_{t}}
\sum_{k=0}^{\infty}
 \frac {(\lambda_{k}(\mu))^{n}\lambda'_{k}(\mu)}{(2\pi)^{n+1}} 
\Big(
|\Lambda_{k}^{
{(\lambda_{k}(\mu))}} 
g(z)|
+
|\Lambda_{k}^{
{-\lambda_{k}(\mu)}} 
g(z)|
\Big) .
\end{align*}
Therefore, the triangle inequality implies
\begin{align*}
\left(
\int\limits_{\CC^{n}}
|{\mathcal P}^{m}_{\mu} f (z,t)|^{q} dz
\right)^{\frac 1q}
&\leq
\frac {\| h\|_{L^{1}(\RR)}}{(2\pi)^{n+1}}
\sum_{k=0}^{\infty}
{\lambda_{k}(\mu)^{n}\lambda'_{k}(\mu)} 
\Big(
\|\Lambda_{k}^{
{\lambda_{k}(\mu)}} 
g\|_{L^{q}(\CC^{n})}
+
\| \Lambda_{k}^{
{-\lambda_{k}(\mu)}} 
g\|_{L^{q}(\CC^{n})}
\Big),
\end{align*}
which by Lemma~\ref{lemmanormaPirok} 
and the subsequent observation,
yields, %(if $\Lambda_{k}$ is bounded from$L^{p}(\CC^{n})$ to $L^{q}(\CC^{n})$),
\begin{align*}
\left(
\int\limits_{\CC^{n}}
|{\mathcal P}^{m}_{\mu} f (z,t)|^{q} dz
\right)^{\frac 1q}
&\leq
\frac 2{(2\pi)^{n+1}}
\| h\|_{L^{1}(\RR)}
\|g\|_{L^{p}(\CC^{n})}
\times \\ &\times
\left(
\sum_{k=0}^{\infty}
{(\lambda_{k}(\mu))^{
n\left(\frac1p - \frac1q
\right)}
\lambda'_{k}(\mu)}
\|\Lambda_{k}^{1}
\|_{L^{p}(\CC^{n}) \rightarrow L^{q}(\CC^{n})}
\right) ,
\end{align*}
proving the statement.
\end{proof}

\medskip

In order to control the convergence 
of the series in
\eqref{stima astratta 11},
we apply 
the sharp estimates for the 
$L^p-L^2$ norms, $1\le p\le 2$, of the operators $\Lambda_{k}$ recently proved by H. Koch and F. Ricci,
stating that
\begin{equation}\label{sogge 1}
%\nu_p=
\|\Lambda_{k}\|_{L^{p}(\CC^{n})
\rightarrow L^{2}(\CC^{n})}
\lesssim
C (2k+n)^{\gamma(\frac1p)}\,,
\quad
 1\le p\le 2\,,
\end{equation}
where $\gamma$
is the piecewise
affine function on $[\frac12,1]$
defined by
\begin{equation*}
%\label{KR stime}
\gamma\left (\frac1p \right):=
\begin{cases}\!
{n\left(\frac1p-\frac{1}{2}\right)-\frac{1}{2} 
}
&\,\text{if
  $1\le p\leq p_{*}$,}\cr
{\frac{1}{2}(\frac{1}{2}-\frac1p)}
&\,\text{if
  $ p_{*}\le p\le 2$,}\cr
\end{cases}
\end{equation*}
with  {{critical point}}
$p_{*} = p_{*}(2n)$,
defined by \eqref{punto-critico}.

\begin{lemma}
If $1 \leq p \leq 2\leq q\leq \infty $, then
%since
%$\frac 1 {q'}- \frac12 =\frac12 - \frac 1 {q} $.
\begin{equation}\label{p,q 1}
\|\Lambda_{k}\|_{L^{p}(\CC^{n})
\rightarrow L^{q}(\CC^{n})}
\le
C (2k+n)^{\gamma\left(\frac1p \right) + \gamma\left(\frac 1 {q'} \right)}\,.
\end{equation}
\begin{comment}
\begin{equation*}
\gamma(s):=
\begin{cases}\!
{\frac{1}{2}\big( \frac12-s\big) 
}
&\,\text{if
  $\frac12 \leq s \leq  \frac{1}{p_*(n)}$,}\cr
{n(s-\frac12) -\frac{1}{2}}
&\,\text{if
  $  \frac{1}{p_*(n)}\leq s \leq 1$,}\cr
\end{cases}
\end{equation*}
where $p_*(n)$ is defined as in \eqref{punto-critico}.
\end{comment}
\end{lemma}

\begin{proof}
By duality the estimates \eqref{sogge 1} are equivalent to
\begin{equation*}%\label{sogge dua}
\|\Lambda_{k}\|_{L^{2}(\CC^{n})
\rightarrow L^{q}(\CC^{n})}
\leq
C (2k+n)^{\gamma\left(\frac1{q'}\right)} \,,
\quad\text{ $2\le q\le \infty$}\,,
\end{equation*}
yielding
\eqref{p,q 1}.
\\
It easily follows
from \cite{KR} that the above estimate
is sharp for $1 \leq p \leq p_*(2n)$,
 $p'_*(2n) \leq q \leq \infty$
and for $p_*(2n)\leq p \leq 2$,
 $2 \leq q \leq p'_*(2n)$.

\end{proof}

In the case of the sublaplacian,
%we have
%$\lambda_{k}(\mu) = \mu/ {(2k+n)}$.
%Hence, 
from \eqref{stima astratta 11}
it follows that
\begin{align}\label{stima astratta 21}
\left
\|\mathcal P^{L}_{\mu} f
\right\|_{L^{\infty}_{t}L^{q}_{z}}
&\leq
C
{\mu^{n \left(\frac 1p - \frac 1q\right)}} 
\left(
\sum_{k=0}^{\infty}
{(2k+n)^{\gamma\left(\frac1p \right) + \gamma\left( \frac 1 {q'} \right)
-n\left(\frac 1p - \frac 1q \right) - 1}
}
\right)
\| f \|_{L^{1}_{t}L^{p}_{z}},
\end{align}
for $1 \leq p \leq 2 \leq q \leq \infty$.
To study the convergence of this series,  we need to distinguish
four cases
according to the relative position
of $p$ and $q$ with respect to
the critical exponents $p_{*}(2n)$ and
$p'_{*}(2n)$.
We collect the result in
the following lemma.

\begin{lemma}\label{convergenza della serie}
 For any real number $\alpha$ we define
%the series
\begin{align*}
\mathcal S_{\alpha} = \sum_{k=0}^{\infty}
{(2k+n)^{\gamma\left(\frac1p \right) + \gamma\left(\frac 1 {q'} \right)
-n\left(\frac 1p - \frac 1q \right) + 
\alpha}
}.
\end{align*}
Then:
\begin{enumerate}[(I)]

\item For $1 \leq p < p_{*}$, $2 \leq q \leq p'_{*}$ the series $\mathcal S_{\alpha}$
converges if
$
\alpha <
\frac{2n+1}{2}
\left(
\frac{1}{p'_{*}} -\frac{1}{q}
\right)$.

\item For $1 \leq p < p_{*}$ and
$p_{*}' \leq q \leq \infty$ the series $\mathcal S_{\alpha}$
converges if $\alpha < 0$.

\item For $p_{*} \leq p \leq 2$ and $2 \leq q
\leq p_{*}'$ the series $\mathcal S_{\alpha}$
converges if
$
\alpha < \frac{2n+1}2 \left( \frac 1p -\frac 1q \right) - 1$.
%12%

\item For $p_{*} \leq p \leq 2$ and $p_{*}' \leq q \leq \infty$ the series $\mathcal S_{\alpha}$
converges if
$
\alpha <\frac{2n+1}2 \left( \frac 1p- 
\frac1{p_{*}}\right)$.
\end{enumerate}
\end{lemma}

\begin{proof}
In order not to burden the exposition,
we prove only (I), the other cases
being analogous.
The series converges if
\begin{equation}\label{esponente 1}
{\gamma\left(\frac1p  \right) +\gamma\left(\frac 1 {q'} \right)
-n\left(\frac 1p - \frac 1q \right)} <
 -1-\alpha.
\end{equation}
If $1 \leq p < p_{*}$, $2 \leq q \leq p'_{*}$ we have
\begin{align*}
{\gamma\left(\frac1p  \right) +\gamma\left(\frac 1 {q'} \right)
-n\left(\frac 1p - \frac 1q \right)}
&=
\frac{2n+1}{2q} - \frac{2n+3}{4} .
\end{align*}
Thus, the condition \eqref{esponente 1}
becomes
\begin{align*}
\alpha <
\frac{2n+1}{2}
\left(
\frac{2n-1}{2{(2n+1)}} -\frac{1}{q}
\right)
=
\frac{2n+1}{2}
\left(
\frac{1}{p'_{*}} -\frac{1}{q}
\right)
\end{align*}
proving (I).
%which may be satisfied since $q < p'_{*}$.
%\end{enumerate}
\end{proof}

%11%

\begin{comment}
Specifically,
in \eqref{stima astratta 21} we have 
$\alpha = -1$,
so from the lemma
 we conclude that:
\begin{enumerate}[(I)]

\item For $1 \leq p < p_{*}$, $2 \leq q \leq p'_{*}$ the series
converges since
$
\alpha = -1 < %0 \leq
\frac{2n+1}{2}
\left(
\frac{1}{p'_{*}} -\frac{1}{q}
\right)$.

\item For $1 \leq p < p_{*}$ and
$p_{*}' \leq q \leq \infty$ the series
converges since $\alpha = -1 < 0$.

\item For $p_{*} \leq p \leq 2$ and $2 \leq q
\leq p_{*}'$ the series
converges for $(p,q) \neq (2,2)$
since, in this case,
$
\alpha = -1 < \frac{2n+1}2 \left( \frac 1p -\frac 1q \right)-1$.
%12%

\item For $p_{*} \leq p \leq 2$ and $p_{*}' \leq q \leq \infty$ the series
converges since
\begin{align*}
\frac{2n+1}2 \left( \frac 1p- 
\frac1{p_{*}}\right)
\geq - \frac12 > \alpha.
\end{align*}
\end{enumerate}

We have therefore proved the following
result.
\end{comment}

Estimate \eqref{stima astratta 21} and Lemma \ref{convergenza della serie}
with $\alpha = -1$ entail the following result.

\begin{theorem}
\label{muller improved 1}
For all pairs $(p,q)$ with $1 \leq p \leq 2$, $2 \leq q \leq
\infty$, and $(p, q) \neq (2, 2)$,
there is a constant $C_{p q}$, such that
for all Schwartz functions $f$ and
all positive numbers $\mu$ we have the inequality
\begin{align}
\notag
\left
\|\mathcal P^L_{\mu} f
\right\|_{L^{\infty}_{t}L^{q}_{z}}
&\leq
C_{pq}
{\mu^{n \left(\frac 1p - \frac 1q\right)}} \| f \|_{L^{1}_{t}L^{p}_{z}}.
\end{align}
\begin{comment}
In particular, we have
\begin{align}
\notag
\left
\|\mathcal P^L_{\mu} f
\right\|_{L^{\infty}_{t}L^{2}_{z}}
&\leq
C_{p2}
{\mu^{n \left(\frac 1p - \frac 12\right)}} \| f \|_{L^{1}_{t}L^{p}_{z}},
\end{align}
for all  $1 \leq p < 2$.
\end{comment}
\end{theorem}

\medskip

\begin{remark}
Note that
in the four cases listed above we can insert a positive power of $k$ in the series
still preserving the convergence.
This amounts to study a derivative
$\partial^{-\alpha}_{t}$
of negative order of
$\mathcal P^{L}_{\mu} f$
and estimate its norm.
In other words, instead of considering
the operator $\mathcal P^{L}_{\mu}$
given by \eqref{Plambda su Hn}, we introduce the operator
\begin{align*}%\label{derivate di P}
\partial_{t}^{-\alpha}&
\mathcal P^{L}_{\mu} f (z,t) =
\frac {\mu^{n-\alpha}} {(2\pi)^{n+1}}\\
&\times \sum_{k=0}^{\infty}
{(n+2k)^{\alpha -(n+1)}}
\left(
{e^{-i { \lambda_{k}(\mu)} t}}
\Lambda_{k}^{\lambda_{k}(\mu)}
f^{({
 {\lambda_{k}(\mu)}
})} (z)
+
{e^{i { \lambda_{k}(\mu)} t}}
\Lambda_{k}^{-\lambda_{k}(\mu)}
f^{({
 {-\lambda_{k}(\mu)}
})} (z)
\right)
\end{align*}
and prove estimates like those in the theorem for $\alpha$ small.
It is easy to see that this yields 
the spectral resolution
of the operator 
$$
%m_{\alpha}(L,T) =
\left(\frac{n+1 -\alpha}{n+1}\right)^{\frac1{n+1-\alpha}}
 |T|^{\frac\alpha{n+1-\alpha}} L.
$$

Then, retracing the argument that lead
to \eqref{stima astratta 21},
we obtain
\begin{align*}%\label{stima astratta 3}
\notag
\left
\| \partial_{t}^{-\alpha}
\mathcal P^L_{\mu} f
\right\|_{L^{\infty}_{t}L^{q}_{z}}
&\leq
C
{\mu^{n \left(\frac 1p - \frac 1q\right)-
\alpha}} \left(
\sum_{k=0}^{\infty}
{(2k+n)^{\alpha+
\gamma\left(\frac1p  \right) +\gamma\left(\frac 1 {q'} \right)
-n\left(\frac 1p - \frac 1q \right) - 1}
}
\right)
\| f \|_{L^{1}_{t}L^{p}_{z}}.
\end{align*}
From this estimate, using Lemma \ref{convergenza della serie},
we obtain the following theorem.
We omit the argument which is similar
to that of the previous theorem.

\begin{theorem}
The estimate
\begin{align}\label{stima con le derivate negative}
\left
\| \partial_{t}^{-\alpha}
\mathcal P^L_{\mu} f
\right\|_{L^{\infty}_{t}L^{q}_{z}}
\leq C
{\mu^{n \left(\frac 1p - \frac 1q\right)
-\alpha}}\| f \|_{L^{1}_{t}L^{p}_{z}}
\end{align}
holds:
\begin{enumerate}[(I)]

\item
When
$
\alpha <
\frac{2n+1}{2}
\left(
\frac{1}{p'_{*}} -\frac{1}{q}
\right)+1$ if $1 \leq p \leq p_{*}$
and $2 \leq q \leq p_{*}'$.

\item
When
$
\alpha <\frac{2n+1}2 \left( \frac 1p- 
\frac1{p_{*}}\right)+1$ if $p_{*} \leq p \leq 2$ and $p_{*}' \leq q \leq \infty$.

\item
When $\alpha < 1$ if $1 \leq p \leq p_{*}$ and
$p_{*}' \leq q \leq \infty$.

\item
When $\alpha < \frac{2n+1}2 \left( \frac 1p -\frac 1q \right)$ if $p_{*} \leq p \leq 2$ and $2 \leq q
\leq p_{*}'$.
\end{enumerate}
\end{theorem}

% It may be zero only in case (IV)
%for $(p,q)=(2,2)$.
%In the first two cases the bound is contained between $\frac 12$ and $1$.
We remark that according to the above theorem  
$\alpha$ need always to be
strictly smaller than $1$.
This is consistent with M\"uller's 
 counterexample showing that
in the estimate \eqref{teorema di muller}
one cannot find nothing better than
the $L^{\infty}$-norm in the central variable.
Indeed, if we had \eqref{stima con le derivate negative} with $\alpha=1$, then
the $t$-antiderivative of $\mathcal P^{L}_{\mu} f$ would be bounded.
%and hence $\mathcal P^{L}_{\mu} f$
%would lie in some $L^{p}$-space with $p$ finite.
%e' vero? In generale no, infatti non  vero per $\sin t$.
\end{remark}

%!%
\medskip

\section{Restriction estimates
for the full Laplacian on the Heisenberg group}

Similar to what we have done so far for the sublaplacian, we now study
the case of the  full Laplacian on $\HH_{n}$.
We  consider here only the estimates
for $q=2$. A more detailed discussion of restriction estimates for the full Laplacian 
in the more general framework of M\'etivier groups may be found in \cite{CaCia3}.

\begin{theorem}
For $1\leq p \leq p_{*}$
we have
\begin{equation}
\label{laplacianopieno1}
\left
\|\mathcal P^{\Delta_{\HH}}_{\mu} f
\right\|_{L^{\infty}_{t}L^{2}_{z}}
\leq
C \mu^{n \left(\frac 1p - \frac 12 \right)-\frac14}\,
\| f \|_{L^{1}_{t}L^{p}_{z}}
\end{equation}
and
for $p_{*} \leq p \leq 2$
we have
\begin{equation}
\label{laplacianopieno2}
\left
\|\mathcal P^{\Delta_{\HH}}_{\mu} f
\right\|_{L^{\infty}_{t}L^{2}_{z}}
\leq
C \mu^{\frac{2n-1}4
\left( \frac 1p - \frac 12
\right)
}\,
\| f \|_{L^{1}_{t}L^{p}_{z}}.
\end{equation}
\end{theorem}

\begin{proof}
Plugging \eqref{lambda per il
laplaciano completo} in %the estimate 
\eqref{stima astratta 11} we obtain
\begin{align*}%\label{stima astratta 11}
\left
\|\mathcal P^{\Delta_{\HH}}_{\mu} f
\right\|_{L^{\infty}_{t}L^{2}_{z}}
&\leq
C \| f \|_{L^{1}_{t}L^{p}_{z}}
\times
\\
&\times
\left(
\sum_{k=0}^{\infty}
\|\Lambda_{k}
\|_{L^{p}(\CC^{n}) \rightarrow L^{2}(\CC^{n})}
\frac{\left(
\sqrt{4 \mu + (2k+n)^{2}}
- 2k-n
\right)^{n\left(\frac1p - \frac12\right)}}{\sqrt{4 \mu + (2k+n)^{2}}}
\right)
\\
&\leq
C 
\mu^{n \left(\frac 1p - \frac 12 \right)}\,
\| f \|_{L^{1}_{t}L^{p}_{z}}
\times
\\
&\times
\left(
\sum_{k=0}^{\infty}
\frac{
\|\Lambda_{k}
\|_{L^{p}(\CC^{n}) \rightarrow L^{2}(\CC^{n})}
}
{
\sqrt{4 \mu + (2k+n)^{2}}
}
\left(
\sqrt{
4 \mu + (2k+n)^{2}}
+ 2k+n
\right)^{-n\left(\frac1p - \frac12\right)}
\right).
\end{align*}
%To discuss the convergence of the series,

We split the sum into the sum over
those $k$ such that
$2k+n\leq 2\sqrt \mu$ and those
such that $2k+n>2\sqrt \mu$.
Then we control the first term, say $I$,
by
\begin{align*}
I \leq C
\mu^{n \left(\frac 1p - \frac 12 \right)}\,
\| f \|_{L^{1}_{t}L^{p}_{z}}
\mu^{-\frac n2
\left(\frac1p - \frac12\right)-\frac12}
\left(
\sum_{2k+n\leq 2 \sqrt \mu}
\|\Lambda_{k}
\|_{L^{p}(\CC^{n}) \rightarrow L^{2}(\CC^{n})}
\right)
\\
\leq C
\mu^{\frac n2
\left(\frac1p - \frac12\right)-\frac12}
\| f \|_{L^{1}_{t}L^{p}_{z}}
\left(
\sum_{2k+n\leq 2 \sqrt \mu}
\|\Lambda_{k}
\|_{L^{p}(\CC^{n}) \rightarrow L^{2}(\CC^{n})}
\right)
\end{align*}
and the second, say $II$, by
\begin{align*}
II
\leq C
\mu^{n \left(\frac 1p - \frac 12 \right)}\,
\| f \|_{L^{1}_{t}L^{p}_{z}}
\left(
\sum_{2k+n\geq 2 \sqrt \mu}
\frac{
\|\Lambda_{k}
\|_{L^{p}(\CC^{n}) \rightarrow L^{2}(\CC^{n})}
}
{
(2k+n)^{1 + n \left( \frac 1p - \frac 12
\right)}
}
\right).
\end{align*}

When $1\leq p\leq p_{*}$
by \eqref{sogge 1} we have
\begin{align*}
I
&\leq C
\mu^{\frac n2
\left(\frac1p - \frac12\right)-\frac12}
\| f \|_{L^{1}_{t}L^{p}_{z}}
\left(
\sum_{2k+n\leq 2 \sqrt \mu}
(2k+n)^{n\left(\frac1p - \frac 12\right)-\frac12}
\right)
\\
&\leq C
\mu^{n
\left(\frac1p - \frac12\right)-\frac14}
\| f \|_{L^{1}_{t}L^{p}_{z}}
\end{align*}
and
\begin{align*}
II 
&\leq C
\mu^{n \left(\frac 1p - \frac 12 \right)}\,
\| f \|_{L^{1}_{t}L^{p}_{z}}
\left(
\sum_{2k+n\geq 2 \sqrt \mu}
(2k+n)^{- \frac 32}
\right)
\\
&\leq C
\mu^{n \left(\frac 1p - \frac 12 \right)-\frac14}\,
\| f \|_{L^{1}_{t}L^{p}_{z}}
\end{align*}
proving \eqref{laplacianopieno1}.

%%%%%%%%%%%%%%%

When $p_{*} \leq p \leq 2$ we have
\begin{align*}
I
&\leq C
\mu^{\frac n2
\left(\frac1p - \frac12\right)-\frac12}
\| f \|_{L^{1}_{t}L^{p}_{z}}
\left(
\sum_{2k+n\leq 2 \sqrt \mu}
(2k+n)^{\frac12 \left( \frac 12
- \frac1p\right)}
\right)
\\
&\leq C
\mu^{\frac {(2n-1)}4\left(\frac1p - \frac12\right)}
\| f \|_{L^{1}_{t}L^{p}_{z}}
\end{align*}
and
\begin{align*}
II 
&\leq C
\mu^{n \left(\frac 1p - \frac 12 \right)}\,
\| f \|_{L^{1}_{t}L^{p}_{z}}
\left(
\sum_{2k+n\geq 2 \sqrt \mu}
(2k+n)^{-1 - \frac{2n+1}2 \left( \frac 1p - \frac 12
\right)}
\right)
\\
&\leq C
\mu^{\frac{2n-1}4
\left( \frac 1p - \frac 12
\right)
}\,
\| f \|_{L^{1}_{t}L^{p}_{z}}
\end{align*}
proving \eqref{laplacianopieno2}.
\end{proof}

\section{Spectral resolution of the
sublaplacian on M\'etivier groups }\medskip

%In this section we consider more general groups.
Let $G$ be a connected, simply connected, 
two-step nilpotent Lie group,
with Lie algebra 
$\mathfrak{g}$.
We denote
the centre of 
$\mathfrak{g}$ by $\mathfrak{z}$
%and by $\mathfrak{v}$ its orthogonal complement  
and
set $\dim \mathfrak{z}=d$.
If $\omega\in \mathfrak{z}^*$, the dual of $\mathfrak{z}$,
we define
$$\mathfrak{g}_{\omega}=
\mathfrak{g}/\text{ker}\,\omega\,.$$
Since $\text{ker}\,\omega$, being a subspace of the centre,
is an ideal in $\mathfrak{g}$,
$\mathfrak{g}_{\omega}$ is a Lie algebra. 
The connected simply connected subgroup of $G$ with Lie algebra $\mathfrak{g}_{\omega}$
will be denoted by %the symbol
$G_{\omega}$. 

Let $\mathfrak v$ be a complement 
of $\mathfrak z$ in $\mathfrak g$.
We assume that
$G$ satisfies a non-degeneracy condition, which is expressed in terms of the bilinear application 
$B_{\omega}{(X,Y)} = \omega([X, Y])$, with $X$, $Y$ in $\mathfrak{v}$ and $\omega$ in $S$. Recall that
$B_{\omega}$ is {\it non-degenerate} if
the space
$\{ V \in \mathfrak{v} : \text{$\omega([V, U]) = 0$
for all $U$ in $\mathfrak{v}$} \}$
is trivial.

\begin{definition}\label{gruppo di Metivier} \cite{Me}
We say that $G$ is a 
{\it{M\'etivier group}} if $B_{\omega}$ is non-degenerate for all $\omega\neq 0$.
\end{definition}

In this case the dimension of $\mathfrak{v}$ is even, say
$\dim \mathfrak{v} = 2n$, and
$G_{\omega}$ is isomorphic to the Heisenberg group $\HH_n$ with Lie algebra $\mathfrak{h}_{n}
= \RR \oplus \mathfrak{v}_{n}$,
$\mathfrak{v}_{n} = \RR^{2n}$.
Moreover, $\mathfrak v$
generates $\mathfrak g$
as a Lie algebra.

%and assume that its Lie algebra 
%$\mathfrak{g}$
%is endowed with an inner product
%$\langle \cdot, \cdot\rangle$.
%$\mathfrak{g}=\mathfrak{v}+\mathfrak{z}$

%%%%%%%%%%%%
%Definizione di gruppi di M\'etivier

Only for notational convenience we introduce an
inner product $\langle \cdot, \cdot \rangle$ on
$\mathfrak g$, with the property that $\mathfrak z$ and
$\mathfrak v$ are orthogonal subspaces.
%We remark that our results are essentially independent of this choice.
Let $| \cdot |$ denote the norm induced
 by
 $\langle \cdot, \cdot\rangle$ on $\mathfrak{z}^*$, the dual of $\mathfrak{z}$. %the space of linear forms on $\mathfrak{z}$,
We call $S$ the unit sphere in 
$\mathfrak{z}^*$, that is,
\begin{equation*}
S:=\{\omega\in\mathfrak{z}^*\,:\,
|\omega|=1
\}\,.
\end{equation*}
For any fixed $\omega$ in $S$
there is an element $Z_{\omega}$ in $\mathfrak{z}$ 
such that
$\omega (Z_{\omega})=1$ and  $|Z_{\omega}|=1$.
Indeed, by definition
%we have
$$
|\omega|=\sup_{|Z|=1}|\omega (Z)|\,
$$
and, since $\omega$ is continuous,
% by the Weierstrass theorem
there exists $Z_\omega\in\mathfrak{z}$
such that
$|Z_\omega|=1$ and 
$\omega (Z_\omega)=1$.
%We fix an orthonormal basis
%$\{Z_1,\ldots, Z_d\}$ of $\mathfrak{z}$ such that $Z_1 = Z_{\omega}$.

The centre of the Lie algebra % $\mathfrak{z}$
decomposes into the sum
\begin{equation}\label{decomposizionecentro}
\mathfrak{z}=\RR Z_{\omega}\oplus
\ker
{\omega}\,.\end{equation}
Observe that for every $Z\in\ker
\omega$
we have $\langle Z_{\omega},Z \rangle =0$.
We shall systematically identify the quotient $\mathfrak{z}{/\text{ker\,}\omega}$
with $\RR Z_{\omega}$.
Then $\RR Z_{\omega}\oplus \mathfrak{v}$
is a Lie algebra isomorphic to $\mathfrak{g}_{\omega}$.

%In this case for every $\omega\in S$ the subgroup
%$G_{\omega}$ is isomorphic to a Heisenberg group.

%%%%%%%%%%%%%%

Since $\mathfrak{g}$ is nilpotent, the exponential map, $\exp : \mathfrak{g} \rightarrow G$, is surjective.
Thus we may
parametrize $G$ by
$\mathfrak{v}\oplus\mathfrak{z}$,
endowing it with the exponential coordinates. More precisely,
we fix a basis $\{Z_1, \ldots, Z_{d},V_1, \ldots, V_{2n}\}$
 of $\mathfrak g$, with $\{Z_1, \ldots, Z_{d}\}$ 
a basis of $\mathfrak{z}$ and
$\{V_1, \ldots, V_{2n}\}$ a basis 
of $\mathfrak{v}$,
 and identify
a  point $g$ of $G$  with the point
$(V,Z)$ in $\RR^{k}\times \RR^d$,
such that
$$
g = \exp (V,Z)=
\exp \Big(
\sum_{j=1}^{2n}v_j V_j
+\sum_{a=1}^{d} z_a Z_a\Big).% = g\,,
$$
%where $Z=(z_1,\ldots,z_{d}) \in \RR^d$ and $V=(v_1,\ldots,v_{2n})\in\RR^{2n}$.
%Moreover, by writing $(V,Z)$ for $\text{exp}(V+Z)$ with $Z\in\mathfrak{z}$ and $V\in \mathfrak{v}$,
In these coordinates the product law 
% in $G$ 
 is given by the Baker-Campbell-Hausdorff
 formula
\begin{equation*}
(V,Z)(V',Z')=
\Big(V+V', Z+Z'+\frac{1}{2}[V,V']\Big)\,,
\end{equation*}
for all $V, V'\in\mathfrak{v}$ and $Z,Z'\in\mathfrak{z}$.

If we denote by $dV$ and $dZ$
the Lebesgue measures
on $\mathfrak{v}$ and $\mathfrak{z}$
respectively,
then the product measure
$dVdZ$ is a left-invariant Haar measure
on $G$. % \cite{Corwin}.
We shall denote by $L^{p}(G)$ the corresponding Lebesgue spaces.

Finally, we call $\cS (G)$ the Schwartz space on $G$, that is, the space of functions $f$
on $G$ such that $f\circ \exp$ belongs to
the usual Schwartz space on the Euclidean space $\mathfrak{g}$.

%\section{The sublaplacian and its spectral decomposition }\medskip

To any vector
in $\mathfrak g$, thought
of as the tangent space to $G$ at
the origin, we associate a left-invariant vector field
on $G$.
If $f\in \cS (G)$,
$V=\sum_{j=1}^{2n}v_j V_j$,
$T=\sum_{a=1}^{d}z_a Z_a$,
we set
\begin{align*}
\tilde{V_j} f (V,T)&=
\frac{d}{ds}f \Big(
(V,T)(sV_j,0)
\Big)\Big|_{s=0}
\\
&=
\frac{\partial f}{\partial v_j}(V,T)
+\frac{1}{2}
\sum_{a=1}^{d}
\langle
Z_a, [V, V_j]\rangle
\frac{\partial f}{\partial z_a}(V,T)\,,
\end{align*}
and
\begin{align*}
\tilde{T_{a}} f (V,T)&=
\frac{d}{ds}f \Big(
(V,T)(0,s T_{a})
\Big)\Big|_{s=0}\\
&=
\frac{\partial f}{\partial z_a}(V,T)\,.
\end{align*}
Then the vectors fields
\begin{equation*}
\tilde{V_j} =
\frac{\partial }{\partial v_j}
+\frac{1}{2}
\sum_{a=1}^{d}
\langle
Z_a, [V, V_j]\rangle
\frac{\partial }{\partial z_a}\,,
\qquad
\tilde{T_{a}} = \frac{\partial }{\partial z_a}\,
\end{equation*}
are left invariant.
%For simplicity we shall often write ${V_j}$ for $\tilde{V_j}$, clearly

In terms of these vectors we define
the sublaplacian
\begin{equation*}\label{subla}
L=-\tilde{V}_1^2-\cdots-\tilde{V}_{2n}^2\,,
\end{equation*}
the Laplacian on the centre
\begin{equation*}%\label{subla}
\Delta_{\mathfrak z} = -\tilde{T}_1^2-\cdots-\tilde{T}_{d}^2\,,
\end{equation*}
and the full Laplacian
\begin{equation*}%\label{subla}
\Delta_{G} = L + \Delta_{\mathfrak z}\,.
\end{equation*}
The operators $L$ and $\Delta_{G}$ are  positive and essentially self-adjoint
on $L^{2}(G)$.
Moreover, since the set of vector fields
$\{\tilde{V_1}, \tilde{V_2},\ldots, \tilde{V}_{2n}\}$
generates $\mathfrak{g}$ as a Lie algebra, $L$ and $\Delta_{G}$ are
hypoelliptic.

\medskip

We will obtain the spectral decompositions of $L$ and $\Delta_{G}$
from those
of the sublaplacian $L_{\HH}$
and of the full Laplacian $\Delta_{\HH}$
on the Heisenberg group $\HH_{n}$,
by means of a partial Radon transform in the central variables.
%of the functions on the group.

\begin{definition}
For any 
function $f$ in
 $\cS (G)$
and for $\omega\in S$,
we set
\begin{equation*}\label{erreomega}
R_{\omega} f(V,t):=
\int_{\{ Z'\in\ker \omega\}}
f(V, t Z_{\omega}+Z')\,dZ'\,,
\end{equation*}
%\noindent
where $dZ'$ denotes the Lebesgue measure on
the hyperplane $\ker \omega$ in $\mathfrak{z}$.
\end{definition}

For each $\omega$ in $S$,
$R_{\omega}f$ is a function
on the subgroup $G_{\omega}$,
which is isomorphic to $\HH_{n}$.
% with Lie algebra $\mathfrak{v} \oplus \RR Z_{\omega}
%\simeq \mathfrak{g}\setminus \ker \omega$.
%\footnote{O sull'algebra $\mathfrak{h}_%\omega$?Si', ma si identificano.}
%The mapping $f\mapsto \{ R_{\omega}f%\}$
% is the partial Radon transform of $f$ %with respect to the central variables.
 According to the Euclidean theory,
 the family of functions
 $\{ R_{\omega}f\}_{\omega \in S}$
 completely determines $f$.

Fix $\omega$ in $S$.
If we choose the basis of $\mathfrak z$ in such a way that $Z_1=Z_\omega$, we have
\begin{equation}\label{Radon di Za}
\int_{\text{ker}\omega}
\frac{\partial f }{\partial z_a}
(V, t Z_{\omega}
+Z')dZ'=0,
\end{equation}
for all $a= 2,\ldots,d$, since
$
f(V,Z)$
vanishes as the norm
${|Z|}$ goes to infinity. 
Moreover,
$$
\int_{\text{ker}\omega}
\frac{\partial f }{\partial t}
(V, t Z_{\omega}
+Z')dZ'
=
\frac{\partial }{\partial t}
\int_{\text{ker}\omega}
f (V, t Z_{\omega}
+Z')dZ'
=
\frac{\partial }{\partial t}
R_{\omega} f(V,t).
$$
Hence, setting $T^{(\omega)} =
\frac{\partial }{\partial t}$, we have
$$
R_{\omega} \big(\tilde{Z_{1}} f\big)(V,t)
=
\frac{\partial }{\partial t}
R_{\omega} f(V,t)
=
T^{(\omega)}
R_{\omega} f(V,t)
$$
and
$$
R_{\omega} \big(\tilde{Z_{a}} f\big)(V,t)
= 0,
$$
for all $a= 2,\ldots,d$.

We have also
\begin{align*}
R_{\omega}\left(
\tilde{V_j} f
\right)(V,t)
&=
\int_{\text{ker}\omega}
\big(\tilde{V_j} f
\big)
(V, t Z_{1}+Z')
dZ'
\\
&=
\left(
\frac{\partial }{\partial v_j}
 R_{\omega}f
 \right)
 (V,t)
+\frac{1}{2}
\langle
Z_\omega, [V, V_j]\rangle
\left(
\frac{\partial }{\partial z_1}
R_{\omega}
f
\right)
(V,t)\\
&\qquad +\frac{1}{2}
\sum_{a=2}^{d}
\langle
Z_a, [V, V_j]\rangle
\int_{\text{ker}\omega}
\frac{\partial f }{\partial z_a}
(V, t Z_{\omega}
+Z')dZ'\,.
\end{align*}
Since
$\omega (Z)=\langle Z_\omega, Z\rangle$
for $Z\in \mathfrak{z}$,
we have $\langle
Z_\omega, [V, V_j]\rangle =
\omega ( [V, V_j] )$.
Then it follows from \eqref{Radon di Za}
\begin{align}\label{romega}
R_{\omega}\left(
\tilde{V_j} f
\right)
(V,t)
=
\left(
\frac{\partial }{\partial v_j}
+\frac{1}{2}
\omega\big(
 [V, V_j]\big)
\frac{\partial }{\partial t}
\right)
\big(R_{\omega}
f\big)
(V,t)\,.
\end{align}
Setting
\begin{equation*}
V_j^{(\omega)}:=
\frac{\partial }{\partial v_j}
+\frac{1}{2}
\omega\big(
 [V, V_j]\big)
\frac{\partial }{\partial t}\,,\qquad{j=1,\ldots,k},
\end{equation*}
\eqref{romega} is equivalent to the
following commutation relation
\begin{equation}\label{def dei campi proiettati}
R_{\omega}\left(
\tilde{V_j} f
\right)
= V_j^{(\omega)}
\left( R_{\omega} f \right),
\end{equation}
for all $f$ in $\cS(G)$.
In other words the vector fields $V_j^{(\omega)}$
are the projections on the group $G_\omega$
of the vectors $\tilde{V_j}$.

The vector fields
$\{V_1^{(\omega)},\ldots, V_{2n}^{(\omega)}\}$ together with
$T^{(\omega)} = \frac{\partial }{\partial t}$
yield a basis of left-invariant %vector
fields on the group $G_\omega$.
% with Lie algebra  $\mathfrak{g}\setminus \ker \omega\simeq \mathfrak{v} \oplus \RR Z_\omega$.
In terms of them %these vectors
 we introduce the sublaplacian
\begin{equation}
\label{sublaomega}
L^{(\omega)}=-\big(V_1^{(\omega)}\big)^2-\ldots-\big(V_{k}^{(\omega)}\big)^2\,,
%=-\big(R_\omega X_1\big)^2-\ldots-\big( R_\omega X_{2n}\big)^2\,,
\end{equation}
 on $G_\omega$.
We then have
\begin{equation}\label{eguaRomega}
R_\omega (Lf)=L^{(\omega)}(R_\omega f),
\end{equation}
for all functions $f\in\cS(G)$.
From this relation
we may easily obtain the spectral 
decomposition of $f$ from that
of $R_{\omega} f$ by means of
the inverse of the Radon transform.

To carry out this plan, we
introduce the  partial Fourier transform in the central variable of a function $g$
in $\cS (G_\omega)$,
\begin{equation}\label{definizioneF1}
\mathfrak{F}_1 g(V;\lambda)=
\int_{-\infty}^{\infty}e^{i\lambda t}
g (V,t)dt\,,
\end{equation}
$\lambda\in\RR$.
An integration by parts then yields
\begin{align*}
\mathfrak{F}_1 \Big ({V_j}^{(\omega)} g\Big) (V;\lambda) &=
\int_{-\infty}^{\infty}e^{i\lambda t}
V_j^{(\omega)} g (V,t)dt
\\
&=
\left(
\frac{\partial }{\partial v_j}
- \frac{i}{2} \lambda
\omega\big(
 [V, V_j]\big)
\right)
\mathfrak{F}_1 g(V;\lambda) .
\end{align*}
Defining
\begin{equation}
\label{campi lambda omega}
V_j^{(\lambda, \omega)}:=
 \frac{\partial }{\partial v_j}
- \frac{i}{2} \lambda
\omega\big(
 [V, V_j]\big)\,,
 \qquad{j=1,\ldots,k},
\end{equation}
we have proved that
\begin{equation}
\label{campi lambda omega e F_{1}}
\mathfrak{F}_1 \left ({V_j}^{(\omega)} g \right) =
V_j^{(\lambda, \omega)} \mathfrak{F}_1 \big ( g\big),
 \qquad{j=1,\ldots,k},
\end{equation}
which implies
\begin{equation*}
\mathfrak{F}_1 \Big ({L}^{(\omega)} g\Big) (V;\lambda)
= \Delta^{(\lambda, \omega)} \mathfrak{F}_1 g (V;\lambda),
\end{equation*}
where
\begin{equation}
\label{lambda twisted laplacian}
\Delta^{(\lambda, \omega)} = - (V_1^{(\lambda, \omega)})^2 \dots - (V_{k}^{(\lambda, \omega)})^2
\end{equation}
is the {\it $\lambda$-twisted Laplacian} associated to the group $G_{\omega}$.

In the Euclidean setting it is well known that the Radon transform factorizes
the Fourier transform.
Correspondingly on the group $G$ we have for $\omega \in S$
and $\lambda \in \RR$
\begin{align}
\notag
\mathfrak{F}_1 \big(
R_\omega f\big)(V;\lambda)
&=
\int_{-\infty}^{\infty}
\int_{\ker\omega}
e^{i\lambda\,
 \omega( tZ_\omega +Z')}
f(V, tZ_\omega+Z')dZ' dt
\\
\label{radon e fourier}
&=\int_{\mathfrak{z}}
e^{i\lambda \, \omega( Z)}
f(V,Z)dZ=
\mathfrak{F}_{\mathfrak{z}}f
(V,\lambda\omega)\,,
\end{align}
where %the symbol
$\mathfrak{F}_{\mathfrak{z}} f$
denotes  the Fourier transform of  $f$ along the central variables.

%If $\eta=\rho \omega\in\mathfrak{z}^*$, with $|\omega|=1$ and $\rho >0$,
Similarly, setting $\eta_a = \eta(Z_a)$ for $\eta \in\mathfrak{z}^*$,
 we obtain
\begin{align*}
\mathfrak{F}_{\mathfrak{z}}
({\tilde{V}_j f})
(V,\eta)
&=
\int_{\mathfrak{z}}
e^{i\eta (Z)}
\tilde{V}_j f
(V, Z)
dZ\\
&=
\frac{\partial }{\partial v_j}
\mathfrak{F}_{\mathfrak{z}}
{f}(V,\eta)
+ \frac{1}{2}
\sum_{a=1}^{d}
\langle
Z_a, [V_j, V] \rangle
\int_{\mathfrak{z}}
e^{i\eta(Z)}
\frac{\partial f}{\partial z_a} (V,Z)
dZ
\\
&=
\frac{\partial }{\partial v_j}
\mathfrak{F}_{\mathfrak{z}}f
(V,\eta)
-\frac{i}{2}
\sum_{a=1}^{d}\eta_a
\langle
Z_a, [  V,V_j]\rangle
(\mathfrak{F}_{\mathfrak{z}}{ f})
{(V,\eta)}
\\
&=
\frac{\partial }{\partial v_j}
\mathfrak{F}_{\mathfrak{z}}f
(V,\eta)
-\frac{i}{2}
\eta
\left ( [  V,V_j] \right )
(\mathfrak{F}_{\mathfrak{z}}{ f})
{(V,\eta)}
\\
&=
V^{\eta}_j 
(\mathfrak{F}_{\mathfrak{z}}{ f})
{(V,\eta)}\,.
\end{align*}
Defining
\begin{equation}
V^{\eta}_j =
%\frac{\partial }{\partial v_j}
%-\frac{i}{2}
%\sum_{a=1}^{d}\eta_a
%\langle Z_a, [V, V_j]\rangle=
\frac{\partial }{\partial v_j}
-\frac{i}{2}
\eta( [V, V_j] ),
\quad
j = 1, \dots, k,
\end{equation}
we have showed that
\begin{equation*}
\mathfrak{F}_{\mathfrak{z}}
({\tilde{V}_j f})
=
V^{\eta}_j 
(\mathfrak{F}_{\mathfrak{z}}{ f}).
\end{equation*}

Writing in polar coordinates $\eta$ as
$\eta = \rho \omega$,
with $\rho \geq 0$ and
$\omega$ in $S$, we clearly have
\begin{equation}
\label{che fatica}
V^{\eta}_j = V^{\rho \omega}_j = V^{\rho, \omega}_j,
\end{equation}
where the vector fields $V^{\rho, \omega}_j$ have been
defined in \eqref{campi lambda omega}.

\begin{definition}
\label{twistLaplacian}
Given $\eta \in \mathfrak{z}^{\star}$,
%$=\rho \omega\in\mathfrak{z}^*$, with $\omega\in S$and $\rho>0$, 
the {\rm{$\eta$-twisted Laplacian}}
of $G$ is
\begin{equation}
\label{etatwistLapl}
\Delta_{\eta}=
-\big(V^{\eta}_1\big)^2-\ldots-
\big(V^{\eta}_{k}\big)^2\,.
\end{equation}
\end{definition}

\begin{lemma}
Take $\eta = \rho \omega$ with
$\omega\in S$
and $\rho>0$.
If 
$L^{(\omega)}$ is the sublaplacian defined by
\eqref{sublaomega}
on
$G_{\omega}$,
we have
\begin{equation}
\label{eqLapl}
\Delta_{\rho\omega}=
\Delta^{(\rho, \omega)}=
(L^{(\omega)})^{\rho}
\,,
\end{equation}
that is, 
the ($\rho \omega$)-twisted Laplacian of $G$, defined by \eqref{etatwistLapl},
coincides 
with the $\rho$-twisted Laplacian,
$\Delta^{(\rho, \omega)}$
 of the group
$G_{\omega}$, defined by \eqref{lambda twisted laplacian}.
\end{lemma}

%\eqref{sublaomega}
\begin{proof}
The first identity simply follows from \eqref{che fatica}
and the definitions
\eqref{lambda twisted laplacian} and
\eqref{etatwistLapl}.

%For the second, %fix  $\rho>0$ and $\omega\in S'$.
%we shall compute
%$\mathfrak{F}_{\mathfrak{z}}(\tilde{V}_j f)(V,\rho\omega)\,$,
%that is, the Fourier transform of 
%$(\tilde{V}_j f){(V,\cdot)}$ 
%along the central variables in %at the point $\rho\omega$.

For the second, using \eqref{radon e fourier},
\eqref{def dei campi proiettati},
and
\eqref{campi lambda omega e F_{1}}, which here give
\begin{align}\notag
%(\mathfrak{F}_{\mathfrak{z}} \tilde{V}_j f)(V,\rho\omega)
%&=
\mathfrak{F}_1 \big(
R_\omega \tilde{V}_j f\big)(V;\rho)
&=
\mathfrak{F}_1 
(V^{(\omega)}_j   R_\omega f)
(V;\rho)
%\\
%\notag
%&
=
V^{\rho\omega}_j
\Big(\mathfrak{F}_1 \big(
R_\omega f
\big)\Big)(V; \rho)
%\\&=
%\Big(V^{\rho\omega}_j\mathfrak{F}_{\mathfrak{z}} f\Big)(V),
\end{align}
we obtain
$$
V^{\rho\omega}_j =
(R_\omega \tilde{V}_j )^{\rho}%\footnote{Controllare.}
=
\frac{\partial }{\partial v_j}
-\frac{i}{2}
\rho \,\omega(\langle
 [V, V_j]\rangle),
 $$
 whence $\Delta^{(\rho, \omega)}=
(L^{(\omega)})^{\rho}$ 
%\eqref{eqLapl}
  follows.\end{proof}

%\medskip

We write, for $j=1,\ldots, 2n$,
\begin{align*}
V^{\rho\omega}_j
&=
\frac{\partial }{\partial v_j}
-{i}{\rho}
\sum_{k=1}^{2n}
v_k\,
\omega
\big( 
[V_k,V_j]\big)\\
&=
\frac{\partial }{\partial v_j}
-{i}{\rho}
\sum_{k=1}^{2n}
v_k
B^{\omega}_{kj}\,,
\end{align*}
where
$B^{\omega}_{jk}=
\omega
\big( 
[V_j,V_k]\big)=
-B^{\omega}_{kj}$
%\footnote{Controllare i segni, diversi da Paolo.}
are the entries of the
matrix $B_{\omega}$, introduced
in %the definition of the M\'etivier groups
Definition~\ref{gruppo di Metivier}.

%By Darboux Theorem, 
Since $B_{\omega}$ is non-degenerate and skew-symmetric, 
there exists a $2n\times 2n$ invertible matrix, $A_{\omega}$, such that
%\begin{equation}
%\label{JJ}A_{\omega}^{t}A_{\omega}= {\II}_{2n}
%\end{equation}and
$$
B_{\omega} \left( A_{\omega} V_1,
A_{\omega} V_2 \right) = J(V_1, V_2),
$$
for all $V_{1}, V_{2}$ in $\mathfrak v$,
where $J(X,Y)$ is the standard 
symplectic form corresponding to the matrix
$$
{\JJ}_{2n}=
{\left(\begin{array}{cc}
                              0 & {{\II}}_n \\
                             {{-\II}}_n & 0 
                           \end{array}\right)}\,,
$$
where $\II_n$ %and $\II_{2n}$
 denotes the $n$-dimensional
%and the $2n$-dimensional
 identity matrix. % respectively.
 It follows that
 \begin{equation*}
(\det A_{\omega})^2 = \frac 1{|\det B_{\omega}|}.
\end{equation*}
Since $\det B_{\omega}$, which is a polynomial function in the components of $\omega$, never vanishes on the unit
sphere,
there is a positive constant $K$ such
that
\begin{equation}\label{controllo dello jacobiano}
\frac 1 K \leq |\det B_{\omega}|^{-\frac 12}
= |\det A_{\omega}|
\leq K,
\end{equation}
for $\omega$ in $S$.

We  introduce a new set of  coordinates ${y_j^\omega}$,
defined in terms of $A_{\omega}$ by
\begin{equation*}
%y_j=:
y_j^\omega :=
\sum_{k=1}^{2n}
(A_{\omega})_{jk}\, v_k\,,
\quad
\text{ for $j=1,\ldots, 2n$},
\end{equation*}
 where 
$(A_{\omega})_{jk}$
denotes the $(j,k)$-entry  
of $A_{\omega}$.

Correspondingly
we  define the  vector  fields
$$
%Y_j^{\mu}:=
Y_j^{\rho\omega}:=
\sum_{k=1}^{2n}
(A_{\omega})_{kj}
 V_k^{\rho\omega}\,,\qquad j=1,\ldots, 2n\,.
$$
When only a point $\omega$ in 
$S$ is considered, we shall 
often suppress the
index $\omega$ from $y_j^\omega$ and $Y_j^{\rho\omega}$,
simply writing $y_j$ and $Y_j^{\rho}$.

One easily checks that
$$Y_j^{\rho}=
\frac{\partial}{\partial y_j}
+i
\rho
\sum_{k=1}^{2n}
y_{k}\, 
\JJ_{jk}\,,\quad j=1,\ldots, 2n\,,
$$
or, more explicitly,
$$Y_k^{\rho}=
\frac{\partial}{\partial y_k}
+i
\rho\,
y_{k+n}\,, \quad
Y_{k+n}^{\rho}=
\frac{\partial}{\partial y_{k+n}}
-i
\rho\,
y_{k}, \quad k=1,\ldots, n\,.
$$
Therefore, the operator
\begin{align*}
%\Delta_{\rho\omega} &= -
\sum_{j=1}^{2n}
\big( Y^{{\rho}}_j\big)^{2}
&= - \sum_{j = 1}^{2n} {{\partial^2}\over{\partial y_j^2}}
+ i \rho \sum_{j = 1}^{n} \Big(y_j {{\partial}\over{\partial y_{j + n}}}
- y_{j + n} {{\partial}\over{\partial y_j}} \Big)
+ \rho^2 \sum_{j = 1}^{2n} y_j^2
\end{align*}
coincides with the %canonical
$\rho$-twisted
Laplacian $\Delta_{\HH}^{\rho}$
associated to the Heisenberg group $\HH_n$.

From now on we shall consider the
coordinates $y_{j}$ as a fixed coordinate system
on $\mathfrak{v}_{n}$, the orthogonal complement
of the centre in $\mathfrak{h}_{n}$,
 interpreting
$A_{\omega}$ as a
family of linear invertible maps from $\mathfrak v$
to $\mathfrak{v}_{n}$ depending on $\omega$.
Correspondingly,
the family of vector
fields $\{Y_{j}^{\rho}\}_{j = 1}^{2n}$
will be thought of as a %fixed 
basis %frame
of vector fields on $\mathfrak{v}_{n}$.
%\footnote{FORSE NON SERVE, se %serve va introdotta
%dopo:
%\begin{definition}
%For any two functions
%$g,h\in L^1 (\mathfrak{v})$
%and any $\eta \in \mathfrak{z}^{*}$,
%we define
%the $\eta$-twisted convolution as
%\begin{equation*}
%\big(g\times_\eta h\big) (U)=
%\int_{\mathfrak{v}} g(U-V)h(V)
%e^{\frac{i}{2} \eta ( [ U, V]) }dV\,,\, U
%\in{\mathfrak{v}}\,.
%\end{equation*}\end{definition}}

\begin{proposition}
Let $\omega\in S$ and $\rho > 0$.
If $g$ is a function in $\cS (\mathfrak v)$,
set $g_{\omega} = g\circ A_{\omega}^{-1}$.
\begin{itemize}
\item[a)] The spectral projection
$\Lambda_k^{\rho \omega}$ 
onto the eigenspace
of the $\rho \omega$-twisted Laplacian,
$\Delta_{\rho \omega}$
(defined by \eqref{etatwistLapl}),
corresponding to the eigenvalue
$\rho (2k+n)$,
is given by
\begin{equation}
\label{proiettorideltaomega}
\Pi_k^{\rho \omega} g
=
\big(g_{\omega}
\times_\rho \varphi_k^\rho\big)\circ A_{\omega}
=
\big(
\Lambda^{\rho}_{k} g_{\omega}
\big)
\circ A_{\omega}\,,
\end{equation}
where $\Lambda^{\rho}_{k}$,
defined by \eqref{Lambda 1},
is the spectral projection
of the $\rho$-twisted Laplacian, $\Delta_{\HH}^{\rho}$, on the Heisenberg group.

\item[b)]
Moreover,
\begin{equation}\label{sviluppo}
g(V)= \rho^n
\sum_{k=0}^{\infty}
\big(
g_{\omega} \times_\rho
\varphi^\rho_k\big)
(A_{\omega} V)
= \rho^n
\sum_{k=0}^{\infty}
\big ( \Pi_{k}^{\rho \omega} g \big ) (V)\,.
\end{equation}
\end{itemize}
\end{proposition}

\begin{proof}
Since %$A_{\omega}^{-1} = A_{\omega}^{t}$,
$$
V_k^{\rho\omega}
=
\sum_{j=1}^{2n}
(A_{\omega})_{kj}
Y^{\rho}_j\,,%^\omega\,,
$$
we have
\begin{equation*}
\Delta_{\rho \omega} \big( g \circ A_{\omega}\big)=
\big(
\Delta^{\rho}_{\HH} g \big)
\circ A_{\omega}\,.
\end{equation*}
%as a consequence of \eqref{etatwistLapl},
Thus, writing $g %= g\circ A_{\omega}^{-1} \circ A_{\omega}
=g_{\omega} \circ A_{\omega}$,
we obtain
\begin{equation*}\label{formulafondamentale}
\Delta_{\rho \omega} g =
\Delta_{\rho \omega} \big( g_{\omega} \circ A_{\omega}\big)=
\big(
\Delta^{\rho}_{\HH} g_{\omega} \big)
\circ A_{\omega}\,,
\end{equation*}
which implies
\begin{align*}%\label{Formula1Paolo}
\Delta_{\rho \omega} \Big( 
\big(
\Lambda^{\rho}_{k} g_{\omega}
\big)\circ A_{\omega}\Big)&=
\Big(\Delta^{\rho}_{\HH}
\big(
\Lambda^{\rho}_{k} g_{\omega}
\big)\Big)\circ A_{\omega}\notag\\
&=
\rho (2k+n)
\big(
\Lambda^{\rho}_{k} g_{\omega}\big)\circ A_{\omega}\,.
\end{align*}

Similarly, we obtain
\eqref{sviluppo}
%Formula \eqref{sviluppo su H_{n}} entails
\begin{equation*}\label{sviluppo1}
g(V)= \rho^n
\sum_{k=0}^{\infty}
\big(
\Lambda^{\rho}_{k} g_{\omega}
\big)
(A_{\omega} V)
= \rho^n
\sum_{k=0}^{\infty}
\big(
g_{\omega} \times_\rho
\varphi^\rho_k
\big)
(A_{\omega} V)\,.
\end{equation*}
This %formula yields
is the
expansion 
of $g$ in terms of the eigenfunctions
of the $\rho\omega$-twisted Laplacian
and the spectral projection
onto the eigenspace
associated to
the eigenvalue $\rho (2k+n)$ of $\Delta^{\rho\omega}$ is
thus given by
\eqref{proiettorideltaomega}.
\end{proof}

\begin{remark}
For each $\omega \in S$
the function $g_{\omega}
= g \circ A_{\omega}^{-1}$ in \eqref{proiettorideltaomega}
may be thought of as a function on the
fixed reference space $\mathfrak{v}_{n}$.
\end{remark}

We shall estimate the norms of the projections
$\Pi^{\rho \omega} _k$
in terms of those of
$\Lambda _k$ by means
of the following
lemma.
%corollary of Lemma \ref{lemmanormaPirok}.

\begin{lemma}\label{lemmanormaPirokomega}
Fix $\omega$ in $S$.
Suppose that $\Lambda_k :L^p (\mathfrak{v}_{n})\to L^q (\mathfrak{v}_{n})$
for some $p,q$.
The following inequality holds
\begin{equation*}
\big\|
\Pi^{\rho \omega}_k g \big\|_{L^q(\mathfrak{v})}
\leq C
%|\det A_{\omega}|^{\frac 1p - \frac 1q}
%|\det B_{\omega}|^{\frac 12 \left(
%\frac 1q-\frac 1 p \right )}
\rho^{n(\frac1p-\frac1q-1)}\big\|
\Lambda_k\big\|_{L^p(\mathfrak{v}_{n})
 \rightarrow
 L^q(\mathfrak{v}_{n})
}
\|g\|_{L^p(\mathfrak{v})}
\end{equation*}
for all $g$ in $\cS (\mathfrak{v})$ and
all $\rho > 0$.
\end{lemma}

\begin{proof}
Changing variables in the integrals,
we obtain
\begin{align*}
\| \Pi ^{\rho \omega}_k
g\|_{L^q (\mathfrak{v})}
&=
\| \Lambda^{\rho}_k
(g_{\omega}) \circ A_{\omega}
\|_{L^q (\mathfrak{v})}
=
|\det A_{\omega}|^{-\frac 1 q}
\| \Lambda^{\rho}_k
(g_{\omega})
\|_{L^q (\mathfrak{v}_{n})}
\\
&\leq |\det A_{\omega}|^{-\frac 1 q}
\| \Lambda^{\rho}_k
\|_{L^p (\mathfrak{v}_{n})\rightarrow L^q (\mathfrak{v}_{n})}
\|
g_{\omega}
\|_{L^p (\mathfrak{v}_{n})}
\\
&=
|\det A_{\omega}|^{\frac 1p-\frac 1 q}
\| \Lambda^{\rho}_k
\|_{L^p (\mathfrak{v}_{n})\rightarrow L^q (\mathfrak{v}_{n})}
\|
g
\|_{L^p (\mathfrak{v})}.
%\\&=
%|\det B_{\omega}|^{\frac 12 \left(
%\frac 1q-\frac 1 p \right )}
%\| \Lambda^{\rho}_k
%\|_{L^q (\mathfrak{v}_{n})\rightarrow L^p (\mathfrak{v}_{n})}
%\|g\|_{L^p (\mathfrak{v})}\,.
\end{align*}
%since $\big|\text{det}\, A_{\omega}\big| = {|\det B_{\omega}|^{-\frac 12}}$.
From \eqref{controllo dello jacobiano}
it then follows that
\begin{equation*}%\label{normeqp}
 \| \Pi^{\rho \omega} _k
\|_{L^p (\mathfrak{v})\rightarrow L^q (\mathfrak{v})}
\leq C
%|\det B_{\omega}|^{\frac 12 \left(
%\frac 1q-\frac 1 p \right )}
\| \Lambda^{\rho}_k
\|_{L^p (\mathfrak{v}_{n})\rightarrow L^q (\mathfrak{v}_{n})}
\,,
\end{equation*}
which
implies the statement
 by Lemma \ref{lemmanormaPirok}.
\end{proof}

\medskip

We are now ready to work out the spectral resolution of the operator
$L$. As in Section 2, to semplify the notation 
we set $\mu_k:=\mu/(2k+n)$.

\begin{theorem}\label{risoluzione-Strichartz}
If $f$ is a Schwartz function on $G$,
then
\begin{equation*}
f(V,Z)=
\int_0^\infty
{\mathcal P}^{L}_{\mu} f (V,Z)d\mu\,,
\end{equation*}
where
\begin{equation}
\label{PLmu}
{\mathcal P}^{L}_{\mu} f (V,Z)=
\mu^{n+d-1}
\sum_{k=0}^{\infty}
(2k+n)^{-n-d}
\left (
\int_{S}
\widehat {h_{\mu_k}}
(\omega)
\big(
\Pi^{\mu_k \omega}_k g \big)
(V)
e^{i\mu_k  \omega (Z)}
d\sigma (\omega)
\right )
\,
\end{equation}
and
\begin{equation*}
L\big(
{\mathcal P}^{L}_{\mu} f \big) =
\mu \,
{\mathcal P}^{L}_{\mu} f\,.
\end{equation*}
\end{theorem}
 
\begin{proof}
For the sake of simplicity,
we suppose that 
$f$ is
a tensor function, 
%the product of $g$ in $\cS (\mathfrak{v})$
%and $h$ in $\cS (\mathfrak{z})$,
that is $f(V,Z)=g(V)h(Z)$
with $g$ and $h$ Schwartz functions.

If $V \in \mathfrak{v}$,
we have
\begin{align*}
\big(\mathfrak{F}_{\mathfrak{z}} f\big)
(V,\rho\omega)&=
%\int_{\mathfrak{z}}e^{-i\rho \,\omega(Z)}f(V,Z)dZ\\
%&=
g(V)
\int_{\mathfrak{z}}e^{-i\rho \,\omega(Z)}h(Z)dZ
\\&=
g(V)
\widehat{h}
(\rho\omega)
=
g(V)
\widehat{h_{\rho}}
(\omega)\,,
\end{align*}
where $\widehat{h}$ is a simpler notation for
the Fourier transform
$\mathfrak{F}_{\mathfrak{z}} h$
of $h$
and
$h_{\rho}(Z)=
\frac{1}{\rho^d}
h\big(\frac{Z}{\rho}\big)$.

From the expansion
\eqref{sviluppo} we then get
\begin{equation*}
\big(\mathfrak{F}_{\mathfrak{z}} f\big)
(V,\rho\omega)=
%\widehat{h} (\rho\omega)\,
%\rho^n \sum_{k=0}^{\infty}
%\big(\Pi_{k}^{\rho \omega} g\big)(V)=
\rho^n
\sum_{k=0}^{\infty}
\widehat{h_{\rho}}(\omega)\,
\big(
\Pi_{k}^{\rho \omega} g
\big)
(V)\,.
\end{equation*}
Plugging this expression into
the inversion formula for the Fourier transform on $\mathfrak{z}$,
written in polar coordinates,
we obtain
\begin{align}\label{Formula3Paolo}
f(V,Z)&=
{1 \over (2 \pi)^{d}}
\int_{0}^{\infty}
\left (
\int_{S}e^{i\rho \omega (Z)}
\big(\mathfrak{F}_{\mathfrak{z}} f\big)
(V,\rho\omega)
d\sigma (\omega)
\right )
\rho^{d-1}d\rho\notag\\
&=
{1 \over (2 \pi)^{d}}
\sum_{k=0}^{\infty}
\int_{0}^{\infty}
\Bigg(
\int_{S}
\widehat{h_{\rho}} (\omega)
\big(
\Pi_{k}^{\rho \omega} g
\big) (V)
e^{i\rho \omega (Z)}
d\sigma (\omega)\Bigg)
\rho^{n+d-1}d\rho\,.
%\notag
%\\
%&={1 \over (2 \pi)^{d}}
%\sum_{k=0}^{\infty}
%\int_{0}^{\infty} \rho^{n+d-1}
%\Bigg( \int_{S} \widehat{h_{\rho}} (\omega)
%\Big( \Lambda^\rho_k (g_{\omega}) \Big)
%(A_{\omega} V) e^{i\rho \omega (Z)}
%d\sigma (\omega)\Bigg)
%d\rho
%\notag
%\\&={1 \over (2 \pi)^{d}}
%\sum_{k=0}^{\infty}
%\int_{0}^{\infty} \rho^{n+d-1}
%\Bigg( \int_{S} \widehat{h_{\rho}} (\omega)
%\big( \Pi^{\rho \omega}_k g \big) (V)
%e^{i\rho \omega (Z)} d\sigma (\omega)\Bigg) d\rho\,,
\end{align}
%\footnote{Controllare l'uso della formula %di inversione.}

Since
$$
\Delta_{\mathfrak z}
\Big(
e^{i\rho \omega (Z)}
\big(\Pi^{\rho \omega}_k g \big) (V)
\Big)=
\rho^{2}
e^{i\rho \omega (Z)}
\big(\Pi^{\rho \omega}_k g \big)
(V)\,,
$$
we see that
$$
\Delta_{\mathfrak z}
\left(
\int_S
\widehat{h_{\rho}} (\omega)
e^{i\rho \omega (Z)}
\big(\Pi^{\rho \omega}_k g \big)
(V)d\sigma (\omega)
\right)
=
\rho^{2}
\int_S
\widehat{h_{\rho}} (\omega)
e^{i\rho \omega (Z)}
\big(\Pi^{\rho \omega}_k g \big)
(V)
d\sigma (\omega)\,.
$$
Similarly, since
$$
L
\Big(
e^{i\rho \omega (Z)}
\big(\Pi^{\rho \omega}_k g \big) (V)
\Big)=
\rho (2k+n)
\left(
e^{i\rho \omega (Z)}
\big(\Pi^{\rho \omega}_k g \big)
(V)
\right)\,,
$$
we have
$$
L
\left(
\int_S
\widehat{h_{\rho}} (\omega)
e^{i\rho \omega (Z)}
\big(\Pi^{\rho \omega}_k g \big)
(V)d\sigma (\omega)
\right)
=
\rho (2k+n)
\left(
\int_S
\widehat{h_{\rho}} (\omega)
e^{i\rho \omega (Z)}
\big(\Pi^{\rho \omega}_k g \big)
(V)
d\sigma (\omega)
\right)\,.
$$

In the sum \eqref{Formula3Paolo}
replacing the eigenvalue
$\rho$ of $\sqrt {\Delta_{\mathfrak z}}$
by the eigenvalue $\mu = \rho (2k+n)$ of $L$ in the integrals, we obtain,
as in the case of the Heisenberg group, the spectral
decomposition of $L$,
\begin{align*}
f(V,Z)&=
\sum_{k=0}^{\infty}
(2k+n)^{-n-d}
\int_{0}^{\infty}
\mu^{n+d-1}
\Bigg(
\int_{S}
\widehat {h_{\mu_k}}
(\omega)
\Big(
\Pi^{\mu_k \omega}_k g \Big)
(V)
e^{i\mu_k  \omega (Z)}
d\sigma (\omega)
\Bigg)
d\mu\,.
\end{align*}
In particular, being
\begin{align*}
L
\left(
\int_S
\widehat {h_{\mu_k}}
(\omega)
\big(
\Pi^{\mu_k \omega}_k g \big)
(V)
e^{i\mu_k  \omega (Z)}
d\sigma (\omega)
\right) =
\mu 
\left(
\int_{S}
\widehat {h_{\mu_k}}
(\omega)
\big(
\Pi^{\mu_k \omega}_k g \big)
(V)
e^{i\mu_k  \omega (Z)}
d\sigma (\omega)
\right)\,,
\end{align*}
%that is, the integral over the unit sphere $S$
%in the dual space  $\mathfrak{z}^{*}$
%is an eigenfunction of $L$ with eigenvalue $\mu$.
if we define ${\mathcal P}^{L}_{\mu}$
as in \eqref{PLmu},
we see that
\begin{equation*}
L\big(
{\mathcal P}^{L}_{\mu} f \big) =
\mu \,
{\mathcal P}^{L}_{\mu} f\,
\end{equation*}
and
\begin{equation*}
f(V,Z)=
\int_0^\infty
{\mathcal P}^{L}_{\mu} f (V,Z)d\mu\,.
\end{equation*}
\end{proof}

\section{Restriction estimates
for M\'etivier groups}

In this section we estimate  the norms
of the operators $\cP^{L}_{\mu}$
in the general framework of 
a M\'etivier group $G$.
As for the Heisenberg group,
we  first prove  a conditional result,
guaranteeing the boundedness
of $\cP^{L}_{\mu}$ from $L^p(G)$ to $L^{q}(G)$
on the assumption that the projections 
$\Lambda_{k}$
of the twisted Laplacian are bounded
from $L^{p}(\mathfrak v_{n})$ to $L^{q}(\mathfrak v_{n})$.

\begin{theorem}\label{condizionale}
Assume that
$1\le r\le 2\frac{d+1}{d+3}$.
If the projections $\Lambda_{k}$
are bounded from $L^{p}(\mathfrak v_{n})$
to $L^{q}(\mathfrak v_{n})$,
with $1 \leq p \leq 2 \leq q \leq \infty$,
then the following inequality holds
\begin{align}\label{stimaconTSuno}
\Big\|
 \mathcal{P}^{L}_\mu f\Big\|_{L^{r'}(\mathfrak{z}) L^q(\mathfrak{v})}&\le C
 \mu^{d(\frac1r-\frac{1}{r'})+n(\frac1p-\frac1q)-1}\notag \\
\qquad&\times
\Big(
\sum_{k=0}^{\infty}
(2k+n)^{-
d(\frac1r-\frac{1}{r'})
-n(\frac1p-\frac1q)}
\big\|
\Lambda_k   
\big\|_{L^p (\mathfrak{v_{n}})\rightarrow L^q (\mathfrak{v_{n}})}
\Big)
\big\|
f
\big\|_{L^r (\mathfrak{z}) L^p (\mathfrak{v})}
\,.\end{align}
\end{theorem}

\begin{proof}
In order to simplify the notation,
 we write $f(V,Z) = h(Z) g(V)$, with $h$ and $g$ Schwartz functions.
However, in the proof we will never
use this fact.
%assume that $f$ is a tensor function.
We take  $\alpha:\mathfrak{v}\to \CC$
and $\beta:\mathfrak{z}\to\CC$,
$\alpha\in\cS (\mathfrak{v})$,
$\beta\in\cS (\mathfrak{z})$.
%and assume  that $\|\alpha\|_{L^{q'}(\mathfrak{v})}=1$.
Then
\begin{align*}
\langle \mathcal{P}^{L}_\mu f,& \alpha\otimes \beta\rangle_{\mathfrak{v}\oplus\mathfrak{z}}=
\int_{\mathfrak{v}}
\int_{\mathfrak{z}}
\overline{\alpha (V)}
\overline{\beta (Z)}
\mu^{n+d-1}
\sum_{k=0}^{\infty}
(2k+n)^{-n-d}
\times
\\
&\qquad\qquad\times
\Bigg(
\int_{S}
e^{i\mu_k  \omega (Z)}
\widehat{h_{\mu_k}}(\omega)
\Big(\Pi^{\mu_k \omega}_k g \Big)(V)
d\sigma (\omega)\Bigg)dZ\,dV\\
&=
\mu^{n+d-1}
\sum_{k=0}^{\infty}
(2k+n)^{-n-d}
\Bigg(
\int\limits_{S}
\int\limits_{\mathfrak v}
\widehat{h_{\mu_k}}(\omega)
\Big(\Pi^{\mu_k \omega}_k g \Big)(V)
\times
\\
&\qquad\qquad\times
\int_{\mathfrak{z}}
e^{i\mu_k  \omega (Z)}
\overline{\alpha(V)}
\overline{\beta (Z)}
dZ
dV
d\sigma (\omega)\Bigg)
\\
&=
\mu^{n+d-1}
\sum_{k=0}^{\infty}
(2k+n)^{-n-d}
\Bigg(
\int\limits_{S}
\left \langle
\widehat{h_{\mu_k}}(\omega)
\big(\Pi^{\mu_k \omega}_k g \big) ,
{\widehat{ \beta_{\mu_k}}(\omega)}
\alpha
\right \rangle_{\mathfrak{v}}
d\sigma (\omega)\Bigg)\,.
\end{align*}
%where we set $A_{\omega} = A_{\omega}$.
%Since $\|\alpha\|_{L^{q'}(\mathfrak{v})}=1$,
Applying the H\"older's inequality
to the inner integral we deduce that
\begin{align*}
\Big|
\langle \mathcal{P}^{L}_\mu f,& \alpha\otimes \beta\rangle_{\mathfrak{v}\oplus\mathfrak{z}}\Big|
\leq \mu^{n+d-1}
\sum_{k=0}^{\infty}
(2k+n)^{-n-d}
\times
\\
&\times
\Bigg(
\int_{S}
\big\|
\widehat{h_{\mu_k}}(\omega)
\big(\Pi^{\mu_k \omega}_k g \big) 
\big\|_{L^q (\mathfrak{v})}
\|
{\widehat{ \beta_{\mu_k}}(\omega)}
\alpha 
\|_{L^{q'}(\mathfrak{v})}
d\sigma (\omega)\Bigg)\,.
\end{align*}

Using Lemma \ref{lemmanormaPirokomega},
we then obtain
\begin{align*}
\Big|
\langle \mathcal{P}^{L}_\mu f, \alpha\otimes \beta\rangle_{\mathfrak{v}\oplus\mathfrak{z}}\Big|
&\leq C
 \mu^{n+d-1}
\sum_{k=0}^{\infty}
(2k+n)^{-n-d}\,
\big\| \Lambda_k\big\|_{L^p(\mathfrak{v}_{n})
 \rightarrow
 L^q(\mathfrak{v}_{n})
}\,
\lambda_{k}(\mu)^{n(\frac1p-\frac1q-1)}
\times
\\
&\times
\Bigg(
\int_{S}
\big\|
\widehat{h_{\mu_k}}(\omega) g 
\big\|_{L^p (\mathfrak{v})}
\|
{\widehat{ \beta_{\mu_k}}(\omega)}
\alpha 
\|_{L^{q'}(\mathfrak{v})}
d\sigma (\omega)\Bigg)
\\
&\leq C
 \mu^{d+{n(\frac1p-\frac1q)}-1}
\sum_{k=0}^{\infty}
(2k+n)^{-n(\frac1p-\frac1q)-d}\,
\big\| \Lambda_k\big\|_{L^p(\mathfrak{v}_{n})
 \rightarrow
 L^q(\mathfrak{v}_{n})
}
\times
\\
&\times
\Bigg(
\int_{S}
\big\|
\widehat{h_{\mu_k}}(\omega) g 
\big\|_{L^p (\mathfrak{v})}
\|
{\widehat{ \beta_{\mu_k}}(\omega)}
\alpha 
\|_{L^{q'}(\mathfrak{v})}
d\sigma (\omega)\Bigg)\,.
\end{align*}
Then the Cauchy-Schwarz inequality implies
\begin{align*}
\Big|
\langle \mathcal{P}^{L}_\mu f, &\alpha\otimes \beta\rangle_{\mathfrak{v}\oplus\mathfrak{z}}\Big|
\leq
C
\mu^{d + n(\frac1p-\frac1q)-1}
\sum_{k=0}^{\infty}
(2k+n)^{-d - n(\frac1p-\frac1q)}
\big\|
\Lambda_k\big\|_{{L^p(\mathfrak{v}_{n})
\rightarrow L^q (\mathfrak{v}_{n})}}
\\
&\qquad\qquad\times
\Bigg(
\int_{S}
\big\|
\widehat{h_{\mu_k}}(\omega) g 
\big\|_{L^p (\mathfrak{v})}^{2}
d\sigma (\omega)\Bigg)^{\frac12}
\Bigg(
\int_{S}
\big \|
{\widehat{ \beta_{\mu_k}}(\omega)}
\alpha 
\big \|^{2}_{L^{q'}(\mathfrak{v})}
d\sigma (\omega)\Bigg)^{\frac12}\,.
\end{align*}

Since $p \leq 2 \leq q$, it follows that
$\frac 2p \geq 1$ and $\frac 2{q'} \geq 1$.
Therefore we can apply to the integrals on the right hand side the Minkowski integral inequality, to attain
\begin{align*}
\Big|
\langle \mathcal{P}^{L}_\mu f, 
&\alpha\otimes \beta\rangle_{\mathfrak{v}\oplus\mathfrak{z}}\Big|
\leq C
\mu^{d + n(\frac1p-\frac1q)-1}
\sum_{k=0}^{\infty}
(2k+n)^{-d - n(\frac1p-\frac1q)}
\big\|
\Lambda_k\big\|_{{L^p(\mathfrak{v}_{n})
\rightarrow L^q (\mathfrak{v}_{n})}}
\\
&\times
\left(
\int\limits_{S}
\left(
\int\limits_{\mathfrak v}
\big|
\widehat{h_{\mu_k}}(\omega) g(V)
\big|^{p}
dV
\right)^{\frac2p}
d\sigma (\omega)
\right)^{\frac12}
\left(
\int\limits_{S}
\left(
\int\limits_{\mathfrak v}
\big|
{\widehat{ \beta_{\mu_k}}(\omega)}
\alpha (V)
\big|^{q'}
dV
\right)^{\frac2{q'}}
d\sigma (\omega)
\right)^{\frac12}
\\
&\leq C
\mu^{d + n(\frac1p-\frac1q)-1}
\sum_{k=0}^{\infty}
(2k+n)^{-d - n(\frac1p-\frac1q)}
\big\|
\Lambda_k\big\|_{{L^p(\mathfrak{v}_{n})
\rightarrow L^q (\mathfrak{v}_{n})}}
\\
&
\left(
\int\limits_{\mathfrak v}
\left(
\int\limits_{S}
\big|
\widehat{h_{\mu_k}}(\omega) g(V)
\big|^{2}
d\sigma (\omega)
\right)^{\frac p2}
dV
\right)^{\frac1p}
\left(
\int\limits_{\mathfrak v}
\left(
\int\limits_{S}
\big|
{\widehat{ \beta_{\mu_k}}(\omega)}
\alpha (V)
\big|^{2}
d\sigma (\omega)
\right)^{\frac{q'}2}
dV
\right)^{\frac1{q'}}\,.
\end{align*}

The Stein-Tomas theorem, that
for $1\leq r\leq p_*(d)$
furnishes
the bound
\begin{align*}
\left\|\widehat{h_{\mu}}
 \right\|_{L^2 ({S})}
 &\leq
 C
\big\|
 h_{\mu}
 \big\|_{L^r ({\mathfrak{z}})}
= C\mu^{-\frac{d}{r'}}
\|h\|_{L^r (\mathfrak{z})}\,,
\end{align*}
may now be pressed into service to give
%if we suppose that $\|\alpha\|_{L^{q'}(\mathfrak{v})}=1$ and
%$\| \beta\|_{L^r(\mathfrak{z})}=1$,
\begin{align*}
\Big|
\langle \mathcal{P}^{L}_\mu f,& \alpha\otimes \beta\rangle_{\mathfrak{v}\oplus
\mathfrak{z}}\Big|\le C
\mu^{d-1+n(\frac1p-\frac1q)}
\mu^{-2\frac{d}{r'}}
\\
&\times
\left(
\sum_{k=0}^{\infty}
(2k+n)^{-\big(d+n(\frac1p-\frac1q)-2\frac{d}{r'}\big)}
\big\|
\Lambda_k   
\big\|_{L^p (\mathfrak{v}_{n})\rightarrow L^q (\mathfrak{v}_{n})}
\right)
\\
&\times
\left(
\int\limits_{\mathfrak v}
\left(
\int\limits_{\mathfrak z}
\big|
{h}(Z) g(V)
\big|^{r}
dZ
\right)^{\frac pr}
dV
\right)^{\frac1p}
\left(
\int\limits_{\mathfrak v}
\left(
\int\limits_{\mathfrak z}
\big|
\beta(Z)
\alpha (V)
\big|^{r}
dZ
\right)^{\frac{q'}r}
dV
\right)^{\frac1{q'}}
\,,
%&=
%C \rho^{d(\frac1r-\frac{1}{r'})+n(\frac1p-%\frac1q)-1}
%\times
%\\
%&\times
%\left(
%\sum_{k=0}^{\infty}
%(2k+n)^{-
%d(\frac1r-\frac{1}{r'})
%-n(\frac1p-\frac1q)}
%\big\|
%\Lambda_k   
%\big\|_{L^p (\mathfrak{v}_{n})\rightarrow %L^q (\mathfrak{v}_{n})}
%\right)
%\times
%\\
%&\qquad\qquad\times
%\left(
%\int\limits_{\mathfrak v}
%\left(
%\int\limits_{\mathfrak z}
%\big|{h}(Z) g(V)\big|^{r}dZ
%\right)^{\frac pr}
%dV
%\right)^{\frac1p}
\end{align*}
%on the assumption that
%$$
%\left(
%\int\limits_{\mathfrak v}
%\left(
%\int\limits_{\mathfrak z}
%\big|
%\beta(Z)\alpha (V)\big|^{r}
%dZ
%\right)^{\frac{q'}r}
%dV
%\right)^{\frac1{q'}} \leq 1.
%$$
%
whence it follows that
\begin{align*}%\label{stimaconTS}
%{stimaconTSuno}
\Big\|
 \mathcal{P}^{L}_\mu f\Big\|_{L^{r'}(\mathfrak{z}) L^q(\mathfrak{v})}&\le C
\mu^{d(\frac1r-\frac{1}{r'})+n(\frac1p-\frac1q)-1} 
\,
\big\|
f
\big\|_{L^r (\mathfrak{z}) L^p (\mathfrak{v})}
\notag 
\\
\qquad&\times
\left(
\sum_{k=0}^{\infty}
(2k+n)^{-
d(\frac1r-\frac{1}{r'})
-n(\frac1p-\frac1q)}
\big\|
\Lambda_k   
\big\|_{L^p (\mathfrak{v}_{n})\rightarrow L^q (\mathfrak{v}_{n})}
\right)\,,
\end{align*}
proving the assertion.
%for all $1\le r\le 2\frac{d+1}{d+3}$.
\end{proof}

%i%
%The next tool to carry out our program
%are again  the Koch-Ricci estimates for 
%$\big\|\Lambda_k   \big\|_{L^p 
%(\mathfrak{v}_{n})\rightarrow L^q (\mathfrak{v}_{n})}$.
%the spectral projections of the twisted Laplacian.

%To carry out our program
%we use now  the Koch-Ricci estimates for the spectral projections of the twisted Laplacian.

Implementing Theorem \ref{condizionale} with
the Koch-Ricci estimates for 
$\big\|
\Lambda_k   
\big\|_{L^p (\mathfrak{v}_{n})\rightarrow L^q (\mathfrak{v}_{n})}
$,
we are finally  able to prove  Theorem \ref{nostro}.

\begin{comment}
\begin{theorem}
Suppose that $1\le r\leq p_{*}(d) = 
2\frac{d+1}{d+3}$. 
Then for all $p, q$ satisfying $1\leq p\leq 2 \leq q \leq \infty$ and
for all Schwartz functions
$f$, we have
\begin{equation}\label{secondainterpolata}
\|\mathcal{P}^{L}_\mu  f\|_{L^q(\mathfrak{v})L^{r'}(\mathfrak{z}) }
\le C
\mu^{d(\frac{2}{r}-1)+n (\frac{1}{p}-\frac1q)-1 }
\|f\|_{L^p (\mathfrak{v})   L^r (\mathfrak{z}) }\,.
\end{equation}
%for all $q$ such that $\frac1q < d( \frac{2}{r}-1)$.
\end{theorem}
\end{comment}

{\it{Proof of Theorem \ref{nostro}.}}
To prove \eqref{secondainterpolata 1} we
 only need to discuss
the convergence of the series in
\eqref{stimaconTSuno},
\begin{align*}%{stimaconTSuno}
\sum_{k=0}^{\infty}
(2k+n)^{-
d(\frac1r-\frac{1}{r'})
-n(\frac1p-\frac1q)}
\big\|
\Lambda_k   
\big\|_{L^p (\mathfrak{v}_{n})\rightarrow L^q (\mathfrak{v}_{n})}\,.
\end{align*}
Comparing this sum with
\eqref{stima astratta 21}, the
corresponding one %on
in the case of
the Heisenberg group, we
see that here we have
the factor
$(2k+n)^{-
d(\frac1r-\frac{1}{r'})}$
in place of
$(2k+n)^{-1}$.
Since $d(\frac1r-\frac{1}{r'}) > 1$ for
$1\le r\le 2\frac{d+1}{d+3}$ 
and $d \geq 2$,
the convergence %of the sum
 is improved.
%when  $d \geq 2$. 
% this improves the convergence of the sum.
This observation alone
 suffices to prove the statement.
 Indeed,
 when $(p,q) \neq (2,2)$
 the corresponding series for the Heisenberg group converges in the prescribed range and, when
 $(p,q)=(2,2)$,
our series converges since 
$d(\frac1r-\frac{1}{r'}) > 1$, 
as previously observed.
\qed
%i%

\begin{remark}
The following example shows that
the range of $r$ cannot be extended. 
% enlarged
% our estimate are best possible
It is manufactured %produced
by mixing the Knapp
counterexample to the Stein-Tomas
theorem
and M\"uller's example,
showing that estimates
between Lebesgue spaces
for the operators $\mathcal P_{\mu}$
are necessarily trivial.
Let $h$ be a Schwartz function on 
$\mathfrak z$ and $g(V) =
\varphi_{0} \left( \frac V {\sqrt n} \right)
= e^{- \frac{|V|^{2}}{4 n} }$.
Define $f(V,Z) = g(V)h(Z)$, then
\begin{equation*}
\mathcal P^L_{1} f(V,Z) =
\frac 1 {n^{n+1}} g(V) \int\limits_{S}
\widehat h \left(\omega \right)
e^{-i \omega(Z)} d\omega =
\frac 1 {n^{n+1}} g(V) (R^{*}R h)(Z),
\end{equation*}
where $R$ and $R^{*}$ are the restriction
and extension operator.
Hence, an estimate
\begin{equation*}
\|\mathcal{P}^{L}_1  f\|_{L^{r'}(\mathfrak{z}) L^q(\mathfrak{v}) }
\le C
\|f\|_{L^r (\mathfrak{z})  L^p (\mathfrak{v})   }\,
\end{equation*}
with $2 \frac{d+1}{d+3} < r$ would violate the sharpness
of the Stein-Tomas theorem which is
guaranteed by the Knapp example,
since it would imply
\begin{equation*}\|R^{*}Rh\|_{L^{r'}(\mathfrak{z}) }
\leq C
\|h\|_{L^r (\mathfrak{z}) }\,,
\end{equation*}
for all Schwartz functions $h$ on 
$\mathfrak z = \mathbb R^{d}$.
\end{remark}

%\newpage


\begin{thebibliography}{CoKlSi}


\bibitem[CCi1]{CaCia}
V. Casarino and P. Ciatti,
Transferring $L^p$ eigenfunction bounds from $S^{2n+1}$ to $h^n$,
 {\em Studia
Math.}
\textbf{194} 
(2009), 23--42.

\bibitem[CCi2]{CaCia3}
V. Casarino and P. Ciatti,
Restriction estimates for the full Laplacian on M\'etivier groups,  {\em Rend. Lincei Mat. Appl., to appear}. 

\bibitem[CCi3]{CaCia2}
V. Casarino and P. Ciatti,
Bochner-Riesz
means for the  sublaplacian on M\'etivier
groups, {\em in preparation.} 





\bibitem[K]{K}
A.  Kaplan, 
Fundamental solutions for a class of hypoelliptic PDE generated by composition of quadratic forms,
{\em Trans. Amer. Math. Soc.} \textbf{258}  (1980), no.1, 147--153.

\bibitem[KoRi]{KR}
H. Koch and F. Ricci, 
Spectral projections for the twisted Laplacian, 
{\em Studia Math.} \textbf{180} (2007), no. 2, 103--110.

\bibitem[M]{Me}
G. M\' etivier, { Hypoellipticit\' e analytique sur des groupes nilpotents de rang 2}, {\em Duke Math. J.} \textbf{47}
 (1980), no. 1, 195--221.
8
\bibitem[Mu1]{Mu}
D. M\"uller, A restriction theorem for the Heisenberg group,
{\em Ann. of Math.} \textbf{131} (1990), no. 3, 567--587.

\bibitem[Mu2]{Mueller2}
D. M\"uller,
On Riesz means of eigenfunction expansions
for the Kohn Laplacian,
{\em J. Reine Angew. Math.}  \textbf{401}    (1989),
113-121.


\bibitem[MuS]{MuS}
D. M\"uller and  A.Seeger, 
 Singular spherical maximal
 operators on a class of two step nilpotent Lie groups,
 {\em Israel J. Math  }
\textbf{141}    (2004), 315--340.  


\bibitem[RRaTh]{RRTh}
P. K. Ratnakumar, R. Rawat, and S. Thangavelu, 
A restriction theorem for the Heisenberg motion group, 
{\em Studia Math.}, {\textbf {126}} (1997), 1--12.




\bibitem[So1]{So1}
C. Sogge, 
Oscillatory integrals and spherical harmonics,
{\em Duke Math. J.} \textbf{53}  (1986)
),  43--65.

\bibitem[So2]{So2}
\bysame, 
Concerning the $L^{p}$ norm of spectral clusters for second-order elliptic operators on compact manifolds,
{\em J. Funct. Anal.} \textbf{77}  (1988)
),  123--138.

\bibitem[St]{Stein}
E. M. Stein,
{\em 
Harmonic Analysis, Real-Variable Methods, Orthogonality and Oscillatory Integrals},  Princeton Univ. Press,
Princeton, 1993.


\bibitem[SteZ]{SZ} 
K. Stempak and J. Zienkiewicz, 
Twisted convolution and Riesz means, 
{\em J. Anal. Math.} {\textbf {76} }(1998), 93--107.


\bibitem[Str]{Str}
R. Strichartz,
Harmonic analysis as spectral theory of Laplacians,
{\em  J. Funct. Anal.} \textbf{87} (1989), no.1, 51--148.

\bibitem[T]{Tay}
M. E. Taylor, {\em Noncommutative harmonic analysis}, American Mathematical Society, Providence, (1986).

%Taylor, Michael E. Noncommutative harmonic analysis. Mathematical Surveys and Monographs, 22. American Mathematical Society, Providence, RI, 1986. xvi+328 pp. ISBN: 0-8218-1523-7

\bibitem[Th1]{Th}
S. Thangavelu,  Restriction theorems for the Heisenberg group. 
{\em J. Reine Angew. Math.} {\textbf {414}} (1991), 51--65.

\bibitem[Th2]{Th2}
\bysame, {\em Harmonic Analysis on the Heisenberg Group}, Birkh\"auser,
Boston, (1998).



\end{thebibliography}
\end{document}